\documentclass[12pt]{article}
\usepackage{amssymb}
\usepackage{mathrsfs}
\usepackage{amsfonts}
\usepackage{graphicx}
\usepackage{mathptmx}
\usepackage{colortbl}
\usepackage{latexsym,amsmath,amssymb,amsfonts,amsthm}

\newtheorem{theorem}{Theorem}[section]
\newtheorem{lemma}[theorem]{Lemma}

\newtheorem{corollary}[theorem]{Corollary}
\newtheorem{proposition}[theorem]{Proposition}
\newtheorem{remark}[theorem]{Remark}

\newtheorem{defn}[theorem]{Definition}

\begin{document}
\setcounter{page}{1}
\title{Extremal Domains on Hadamard manifolds}
\author{Jos\'{e} M. Espinar$^{\dag}$ \quad and \quad Jing Mao$^{\dag,\sharp}$}
\date{}
\protect\footnotetext{\!\!\!\!\!\!\!\!\!\!\!\!{  MSC 2010: 35Nxx;
53Cxx. }
 }
\maketitle ~~~\\[-15mm]

\begin{center}
{\footnotesize $^{\dag}$Instituto Nacional de Matem\'{a}tica Pura e Aplicada, 110 Estrada Dona Castorina, Rio de Janeiro, 22460-320, Brazil \\
Email: jespinar@impa.br\\
 $^{\sharp}$Department of Mathematics, Harbin Institute of Technology at Weihai, Weihai, 264209, China \\
Email: jiner120@163.com, jiner120@tom.com}
\end{center}


\begin{abstract}
We investigate the geometry and topology of extremal domains in a manifold with negative sectional curvature. An extremal domain is a domain that supports a positive solution to an overdetermined elliptic problem (OEP for short). We consider two types of OEPs.

First, we study narrow properties of such domains in a Hadamard manifold and characterize the boundary at infinity. We give an upper bound for the Hausdorff dimension of its boundary at infinity and how the domain behaves at infinity. This shows interesting relations with the Singular Yamabe Problem.

Later, we focus on extremal domains in the Hyperbolic Space $\mathbb H ^n$. Symmetry and boundedness properties will be shown. In certain sense, we extend Levitt-Rosenberg's Theorem \cite{LR} to OEPs, which suggests a strong relation with constant mean curvature hypersurfaces in $\mathbb H ^n$. In particular, we are able to prove the Berestycki-Caffarelli-Nirenberg Conjecture under certain assumptions either on the boundary at infinity of the extremal domain or on the OEP itself.

Also a height estimate for solutions on extremal domains in a Hyperbolic Space will be given.
 \end{abstract}

\markright{\sl\hfill  J.M. Espinar and J. Mao  \hfill}

\section{Introduction}
\renewcommand{\thesection}{\arabic{section}}
\renewcommand{\theequation}{\thesection.\arabic{equation}}
\setcounter{equation}{0} \setcounter{maintheorem}{0}

Alexandrov \cite{a} introduced \emph{the moving plane method} and used it to prove a very classical result in the theory of constant mean curvature (CMC for short) hypersurfaces:  {\it the only compact CMC hypersurfaces embedded in the Euclidean $n$-space $\mathbb{R}^{n}$ are spheres}. By also applying the moving plane method and meanwhile
improving the boundary point maximum principle to a more delicate version (cf. \cite[Lemma 1]{s}), Serrin \cite{s} proved that if the OEP
\begin{eqnarray} \label{1.1}
\left\{
\begin{array}{llll}
\Delta{u}=-1  \quad&\mathrm{in}\quad ~~\Omega,\\
u>0  \quad&\mathrm{in}\quad ~~\Omega,\\
u=0  \quad&\mathrm{on}\quad \partial\Omega,\\
\langle\nabla{u},\vec{v}\rangle_{\mathbb{R}^{n}}=\alpha
\quad&\mathrm{on}\quad
\partial\Omega,
\end{array}
\right.
\end{eqnarray}
has a solution $u\in{C}^{2}(\overline{\Omega})$, then $\Omega$ must be a ball, where $\Omega$ is a bounded open connected domain in $\mathbb{R}^{n}$, $\vec{v}$ the unit outward normal vector of the boundary $\partial\Omega$, $\langle\cdot \cdot\rangle_{\mathbb{R}^{n}}$ the inner product in $\mathbb{R}^{n}$, and $\alpha$ a non-positive constant. This result of Serrin is of great importance, since it made the moving plane method available to a large part of the mathematical community. If the constant $-1$ in the first equation of the above OEP (\ref{1.1}) is replaced by a function $f$ with  Lipschitz regularity, Pucci and Serrin \cite{ps} can also get the symmetry result, i.e., the domain $\Omega$ must be a ball in $\mathbb{R}^{n}$ also. The OEP has wide applications in physics, which can be used to describe some physical phenomenons. For
instance, if the constant $-1$ in (\ref{1.1}) is replaced by some constant $-k$ depending on the viscosity and the density of a viscous incompressible fluid moving in straight parallel streamlines through a straight pipe of given cross sectional form $\Omega$, and
moreover, if we set up rectangular coordinates $(x,y,z)$ with the $z$-axis directed along the pipe, then the velocity $u$ of this flow satisfies the equation
\begin{eqnarray*}
\Delta{u}+k=0
\end{eqnarray*}
with the boundary condition $u=0$ on $\partial\Omega$. Applying Serrin's result, we can claim that the tangential stress per unit area on the pipe wall, which is represented by $\mu(\nabla{u},\vec{v})$ where $\mu$ is the viscosity, is the same at all points of the wall if and only if it has a circular cross section. Besides, in the linear theory of torsion of a solid straight bar of cross section $\Omega$, and also in the Signorini problem, the OEP introduced above is related to the physical models therein (see, e.g., \cite{jf,iss} for the details).

We know that if one imposes suitable conditions on the separation interface of the variational structure, overdetermined boundary conditions naturally appears in free boundary problems (see, for instance, \cite{ac}). In this process, several methods based on blow-up techniques applied to the intersection of $\Omega$ with a small ball centered at a point of $\partial\Omega$ were used to locally study the regularity of solutions of free boundary problems. This leads to the study of an elliptic equation in an unbounded
domain. In this situation, Berestycki, Caffarelli and Nirenberg \cite{bcn} considered the following OEP
\begin{eqnarray} \label{1.2}
\left\{
\begin{array}{llll}
\Delta{u}+f(u)=0  \quad&\mathrm{in}\quad ~~\Omega,\\
u>0  \quad&\mathrm{in}\quad ~~\Omega,\\
u=0  \quad&\mathrm{on}\quad \partial\Omega,\\
\langle\nabla{u},\vec{v}\rangle_{\mathbb{R}^{n}}=\alpha
\quad&\mathrm{on}\quad
\partial\Omega,
\end{array}
\right.
\end{eqnarray}
where $\Omega\subset\mathbb{R}^{n}$ is an unbounded open connected domain, $f$ is a Lipschitz function. They proved that if furthermore $\Omega$ is a Lipschitz epigraph with some suitable control at infinity, and the above OEP has a bounded smooth solution, then $\Omega$ is a half-space. Also in this paper, they gave a very nice conjecture as follows.

\begin{quote}
{\bf BCN-Conjecture: } {\it If $f$ is a Lipschitz function on $\mathbb{R}_{+}$, and $\Omega$ is a smooth domain in $\mathbb{R}^{n}$ such that $\mathbb{R}^{n}\backslash\overline{\Omega}$ is connected, then the existence of a bounded solution to OEP (\ref{1.2}) implies that $\Omega$ is either a ball, a half-space, a generalized cylinder $B^{k}\times\mathbb{R}^{n-k}$, where $B^{k}$ is a ball in
$\mathbb{R}^{k}$, or the complement of one of them.}
\end{quote}

BCN-Conjecture has motivated many interesting works. For instance, A. Farina and E. Valdinoci \cite{fv1,fv2,fv3} obtained some natural assumptions to conclude that if $\Omega$ is an epigraph where there exists a solution to OEP (\ref{1.2}) then, under those assumptions, $\Omega$ must be a half-space and $u$ is a function of only one variable. When $f$ is a linear function $f(t)=\lambda{t}$, $t>0$, and $n\geqslant3$, by constructing a periodic perturbation of the straight cylinder $B_{1}^{n}\times\mathbb{R}$, where $B_{1}^{n}$ is the unit ball of $\mathbb{R}^{n}$, that supports a periodic solution to OEP (\ref{1.2}), P. Sicbaldi \cite{sp} successfully gave a counterexample to the BCN-Conjecture in dimension greater than or equal to 3. Although the BCN-Conjecture is invalid for $n\geqslant3$, the $2$-dimensional case is still an open problem. Recently, A. Ros and P. Sicbaldi \cite{rs} have given a partial answer to the BCN-Conjecture in
the case of dimension $2$. More precisely, they proved that if $\Omega$ is contained in a half-plane and $|\nabla{u}|$ is bounded, or if there exists a positive constant $\lambda$ such that $f(t)\geqslant\lambda{t}$ for all $t>0$, then the BCN-Conjecture is true for $n=2$. Besides, A. Ros and P. Sicbaldi \cite{rs} have also shown that some classical results in the theory of CMC hypersurfaces extend to the context of OEPs (see \cite[Theorems 2.2, 2.8 and 2.13]{rs}).

From the above discussion, we know that the OEP is an interesting and important topic, which is worthy of investigating and still has some unsolved problems left.

The purpose of this paper is to study the geometry and the topology of a domain $\Omega\subset{M}$, where $M$ is an $n$-dimensional ($n\geqslant2$) manifold with negative sectional curvature, on which the OEP (\ref{1.3}) or (\ref{1.4}) below can be solved. For convenience, we call such domain $\Omega$ to be the \emph{$f$-extremal domain} of the OEP (\ref{1.3}) or (\ref{1.4}).

In this paper, we first consider the following OEP
\begin{eqnarray} \label{1.3}
\left\{
\begin{array}{llll}
\Delta{u}+f(u)=0  \quad&\mathrm{in}\quad ~~\Omega,\\
u>0  \quad&\mathrm{in}\quad ~~\Omega,\\
u=0  \quad&\mathrm{on}\quad \partial\Omega,\\
\langle\nabla{u},\vec{v}\rangle_{M}=\alpha \quad&\mathrm{on}\quad
\partial\Omega,
\end{array}
\right.
\end{eqnarray}where $\Omega$ is an open connected domain in a complete Hadamard $n$-manifold $(M,g)$ with boundary $\partial\Omega$ of class $C^{2}$, $f$ is a given Lipschitz function, $\langle\cdot,\cdot\rangle_{M}$ is the inner product on $M$ induced by the metric $g$, $\vec{v}$ the unit outward normal vector of the boundary $\partial\Omega$ and $\alpha$ a non-positive constant.

In Section 2, we prove narrow properties for the $f$-extremal domain $\Omega\subset{M}$ of the OEP (\ref{1.3}), provided the function $f$ satisfies a property $\mathbb{P}_{1}$ described in Proposition \ref{proposition1} in Section \ref{section2}.

\begin{quote}
{\bf Theorem \ref{ThBoundary}. } {\it Let $\Omega$ be an open (bounded or unbounded) connected domain of an $n$-dimensional ($n\geqslant2$) Hadamard manifold $M$ whose sectional curvature $K$ of $M$ is pinched as follows
\begin{eqnarray*}
-k_{1}\leqslant{K} \leqslant{-k_{2} }< 0,
\end{eqnarray*}
with $k_{1}$ and $k_{2}$ two nonnegative constants. Assume that one can find a (strictly) positive function
$u\in{C}^{2}(\Omega)$ that solves the equation
\begin{eqnarray*}
\Delta{u}+f(u)=0  \quad\mathrm{in}\quad\Omega,
\end{eqnarray*}
where $f:(0,+\infty)\rightarrow\mathbb{R}$ satisfies the property $\mathbb{P}_{1}$ mentioned in Lemma \ref{proposition1} for some constant $\lambda$ satisfying $\lambda>\frac{(n-1)^{2}k_{1}}{4}$.

Then, there is no conical point $x \in \partial _{\infty} \Omega$ of radius $r > \frac{c_{1}(n,k_1)}{\sqrt{\lambda}}$. In particular, the Hausdorff dimension of $\partial _{\infty} \Omega$ satisfies ${\rm dim}_{\mathcal H}\left( \partial _{\infty} \Omega \right) < n-1  $. Here, $c_{1}(n,k_1)$ is a uniform constant depending only on $n$ and $k_1$.}
\end{quote}

Loosely speaking, $\mathbb{P}_{1}$ means that we can find a solution $v$ to the Dirichlet problem
\begin{eqnarray*}
\left\{
\begin{array}{llll}
\Delta{v}=-\lambda v  \quad&\mathrm{in}\quad ~~B(p,R),\\
v>0  \quad&\mathrm{in}\quad ~~ B(p,R),\\
v=0  \quad&\mathrm{on}\quad \partial B(p,R),\\
\end{array}\right.
\end{eqnarray*}in a ball $B(p,R) \subset M$ of certain radius, bounded uniformly by $\frac{c_{1}(n,k_1)}{\sqrt{\lambda}}$, for any point $p \in M$.

Geometrically speaking, a conical point $x \in \partial _{\infty} \Omega$ of radius $r$ (see Definition \ref{DefConical}) means that $\Omega$ contains a neighborhood at infinity of the set of points at fixed distance $r$ from a complete geodesic $\gamma $ in $M$.  Then, as an immediate consequence we get

\begin{quote}
{\bf Corollary \ref{CorBoundary}. } {\it If $f$ satisfies property $\mathbb{P}_{1}$ mentioned in Lemma \ref{proposition1} for some constant $\lambda$ satisfying $\lambda>\frac{(n-1)^{2}k_{1}}{4}$, then a horoball can not be a $f$-extremal domain in a Hadamard manifold $M$ of sectional curvature $-k_{1}\leqslant{K} \leqslant{-k_{2} }< 0$.}
\end{quote}

In Section 3, we focus on a more general OEP, that is, we consider the OEP
\begin{eqnarray} \label{1.4}
\left\{
\begin{array}{llll}
\sum\limits_{i=1}^{n}a_{i}(u, |\nabla{u}|)\cdot{\partial^{2}_{ii}u}+f(u,|\nabla{u}|)=0  \quad&\mathrm{in}\quad ~~\Omega,\\
u>0  \quad&\mathrm{in}\quad ~~\Omega,\\
u=0  \quad&\mathrm{on}\quad \partial\Omega,\\
\langle\nabla{u},\vec{v}\rangle_{\mathbb{H}^{n}(-k)}=\alpha
\quad&\mathrm{on}\quad
\partial\Omega,
\end{array}
\right.
\end{eqnarray}
where $\Omega$ is an open (bounded or unbounded) connected domain, with boundary $\partial\Omega$ of class $C^{2}$, in the hyperbolic $n$-space $\mathbb{H}^{n}(-k)$ with constant sectional curvature $-k<0$, $\langle\cdot,\cdot\rangle_{\mathbb{H}^{n}(-k)}$ is the inner product on $\mathbb{H}^{n}(-k)$, and $\alpha$, $\vec{v}$ have the same meanings as those in the OEP (\ref{1.2}). Moreover, $a_{i}(u,|\nabla{u}|)$ and $f(u,|\nabla{u}|)$ are \emph{continuously differentiable} functions with respect to variables $u$ and $|\nabla{u}|$, with $|\nabla{u}|$ the norm of the gradient vector $\nabla{u}=(\partial _1 u , \ldots , \partial _n u)$, respectively. Here we have used a convention that for a local
coordinate system $\{x_{i}\}_{1\leqslant{i,j}\leqslant{n}}$ on $\mathbb{H}^{n}(-k)$, $\partial_{i}u$ stands for the partial derivative of $u$ in the $x_{i}$-direction, and then naturally, $\partial_{i}u=\frac{\partial{u}}{\partial{x_{i}}}$ and $\partial ^{2}_{ij}u=\frac{\partial^{2}u}{\partial{x_{i}}\partial{x_{j}}}$ denote the covariant derivatives. Besides, we require that the first PDE in (\ref{1.4}) is uniformly elliptic, that is, there exist positive constants $0<\Lambda_{1}<\Lambda_{2}$ such that
\begin{eqnarray*}
\Lambda_{1}\cdot|\zeta|^{2}\leqslant\sum\limits_{i=1}^{n}a_{i}(u,|\nabla{u}|)\zeta^{2}_{i}\leqslant\Lambda_{2}\cdot|\zeta|^{2},
\end{eqnarray*}
where $\zeta=(\zeta_{1},\cdots,\zeta_{n})\in\mathbb{R}^{n}$.

\begin{remark}  \label{addr}
\rm{We claim that the first PDE in (\ref{1.4}) is well defined, which is equivalently said that the operator $\mathcal {F}u:=\sum_{i=1}^{n}a_{i}(u, |\nabla{u}|)\cdot{\partial^{2}_{ii}u}+f(u,|\nabla{u}|)$ is independent of the choice of the local coordinate system $\{x_{i}\}_{1\leqslant{i,j}\leqslant{n}}$ on $\mathbb{H}^{n}(-k)$.

In fact, set diagonal matrix $A=\left(a_{i}(u,|\nabla{u}|) \delta _{{ij}} \right)_{n\times n}$, $\delta_{ij}$ are the Kronecker symbols, and then we can rewrite $\mathcal {F}u$ as
\begin{eqnarray*}
\mathcal {F}u=\mathrm{Tr}\left(A\nabla^{2}u\right)+f(u,|\nabla{u}|),
\end{eqnarray*}where $\mathrm{Tr}(\cdot)$ denotes the trace of a given matrix, and $\nabla^{2}u$ is the Hessian of $u$.

Clearly, $\mathrm{Tr}\left(A\nabla^{2}u\right)$ is a well defined operator which is independent of the choice of coordinates. Therefore, $\mathcal {F}u$ is globally defined on $\mathbb{H}^{n}(-k)$, and the first PDE in (\ref{1.4}) makes sense.}
\end{remark}

Symmetry and boundedness properties related to the $f$-extremal domain $\Omega$ of the OEP (\ref{1.4}) will be given in Sections \ref{section3}. In certain sense, the Hyperbolic geometry imposes more restrictions to the extremal domain than the Euclidean geometry. We prove:

\begin{quote}
{\bf Theorem \ref{ThLR}.} {\it Assume that $\Omega$ is a connected open domain in $\mathbb{H}^{n}$, with properly embedded $C^{2}$ boundary $\Sigma$, on which the OEP (\ref{1.4}) has a solution $u\in{C}^{2}(\overline{\Omega})$.

Assume that $\partial _{\infty} \Omega \subset E $, where $E$ is an
equator at the boundary at infinity $\mathbb H^{n}_{\infty} =
\mathbb S^{n-1}$. Let $P$ be the unique totally geodesic hyperplane
whose boundary at infinity is $E$, i.e., $\partial _{\infty} P = E$.

It holds:
\begin{itemize}
\item If $\partial _{\infty} \Omega = \O$, then $\Omega$ is a geodesic ball and $u$ is radially symmetric.

\item If $\partial _{\infty} \Omega \neq \O$, then $\Omega$ is invariant by the reflection $\mathscr{R}_{P}$ through $P$ , i.e., $\mathscr{R}_{P} (\Omega) =\Omega$. Moreover, $u$ is invariant under $\mathcal R$, that is, $u(p) = u(\mathcal R (p))$ for all $p \in \Omega$.
\end{itemize}}
\end{quote}

As we pointed out above,  A. Ros and P. Sicbaldi \cite{rs} showed that when the extremal domain is contained in the Euclidean Space $\Omega \subset \mathbb R ^{n}$, there exists a close relation between OEP and properly embedded CMC hypersurfaces in $\mathbb R ^{n}$. They showed analogous results to those for properly embedded CMC hypersurfaces in the Euclidean Space developed by Korevaar-Kusner-Meeks-Solomon \cite{KKMS,KKS,Me}. In the Hyperbolic setting, Theorem \ref{ThLR} could be seen as the extension of Levitt-Rosenberg's Theorem \cite{LR} for OEP.

We mention here two important consequences of Theorem \ref{ThLR}. The first one can be seen as the OEP version of the famous do Carmo-Lawson Theorem \cite{dCL}.

\begin{quote}
{\bf Theorem \ref{theorem3.6-1}.}  {\it Assume that $\Omega$ is a domain in $\mathbb{H}^{n}$, with boundary a $C^2$ properly embedded hypersurface $\Sigma$ and whose asymptotic boundary is a point $x \in \partial _\infty \mathbb H ^n$, on which the OEP (\ref{1.4}) has a solution $u\in{C}^{2}(\overline{\Omega})$.

Then, $\Omega$ is a horoball $D_x (t)$, for some $t\in \mathbb R$ and $u$ is horospherically symmetric.}
\end{quote}

See Definition \ref{Def:HoroSymmetric} for a precise definition of horospherically symmetric. And the OEP version of the Hsiang's Theorem \cite{Hs}.

\begin{quote}
{\bf Theorem \ref{ThDelaunay}. } {\it  Assume that $\Omega$ is a
domain in $\mathbb{H}^{n}$, with boundary a $C^2$ properly embedded
hypersurface $\Sigma$ and whose asymptotic boundary consists in two
distinct points $x, y \in \mathbb S ^{n-1}$, $x\neq y$, on which the
OEP (\ref{1.4}) has a solution $u\in{C}^{2}(\overline{\Omega})$.
Then $\Omega$ is rotationally symmetric with respect to the axis
given by the complete geodesic $\beta $ whose boundary at infinity
is $\{ x,y\}$, i.e., $\beta ^+ = x$ and $\beta ^- = y$. In other
words, $\Omega$ is invariant by the one parameter group of rotations
in $\mathbb H^n$ fixing $\beta$. Moreover, $u$ is axially symmetric
w.r.t. $\beta$.}
\end{quote}

Note that Theorem \ref{ThLR} and Theorem \ref{OneEnd} prove the BCN-conjecture in $\mathbb H ^n$ under assumptions on its boundary at infinity. That is,

\begin{quote}
{\bf Corollary \ref{CorBCNn}.} {\it Assume that $\Omega$ is a domain
in $\mathbb{H}^{n}$, with boundary a $C^2$ properly embedded
hypersurface $\Sigma$ and whose asymptotic boundary consists at most in
one point $x\in \mathbb S ^{n-1}$, on which the OEP (\ref{1.4}) has
a solution $u\in{C}^{2}(\overline{\Omega})$. Then,
\begin{itemize}
\item either $\Omega$ is a geodesic ball and $u$ is radially symmetric,
\item or $\Omega$ is a horoball and $u$ is horospherically symmetric.
\end{itemize}}
\end{quote}

In Section $4$, we prove the BCN-conjecture in dimension $n=2$ under assumptions on the OEP. Specifically:

\begin{quote}
{\bf Theorem \ref{BCNConjecture}.} {\it Let $\Omega \subset \mathbb H^2$ a domain with properly embbeded connected $C^2$ boundary such that $\mathbb H^2 \setminus \overline{\Omega}$ is connected. If there exists a (strictly) positive function $u\in{C}^{2}(\Omega)$ that solves the equation
\begin{eqnarray*}
\left\{
\begin{array}{llll}
\Delta{u}+f(u)=0  \quad&\mathrm{in}\quad ~~\Omega,\\
u>0  \quad&\mathrm{in}\quad ~~\Omega,\\
u=0  \quad&\mathrm{on}\quad \partial\Omega,\\
\langle\nabla{u},\vec{v}\rangle_{\mathbb H ^2}=\alpha \quad&\mathrm{on}\quad
\partial\Omega,
\end{array}
\right.
\end{eqnarray*}where $f:(0,+\infty)\rightarrow\mathbb{R}$ is a Lipschitz function that satisfies the property $\mathbb{P}_{1}$ mentioned in Lemma \ref{proposition1} for some constant $\lambda$ satisfying $\lambda>\frac{1}{4}$, then $\Omega$ must be a geodesic ball and $u$ is radially symmetric.}
\end{quote}

We finish by obtaining a height estimate in Section \ref{section4}.

\begin{remark}
To finish, we would like to point out that, if furthermore $a_{i}$ and $f$ are \emph{analytic}, then the uniqueness of the solution (if exists) to the OEP (\ref{1.4}) can be assured by applying \cite[Theorem 5]{fv3}.
\end{remark}


\section{Narrow properties of $f$-extremal domains}  \label{section2}
\renewcommand{\thesection}{\arabic{section}}
\renewcommand{\theequation}{\thesection.\arabic{equation}}
\setcounter{equation}{0} \setcounter{maintheorem}{0}

We begin this section by proving that a $f$-extremal domain cannot
be too big in $\mathbb H ^n (-k)$ under certain conditions on $f$,
that is, a $f$-extremal domain $\Omega \subset \mathbb H ^n (-k)$
does not contain a ball of radius $R$, $R$ depends only on $k$ and
$n$. We extend this result to Hadamard manifolds. We continue by
studying the boundary at infinity of a $f$-extremal domain in a
Hadamard manifold and how is the behavior of the points at the
boundary at infinity. In particular, we show that a horoball cannot
be a $f$-extremal domain. Finally, we exhibit some interesting
analogies with the singular Yamabe Problem.

\subsection{The Narrow property of $f$-extremal domains}

We first recall some fundamental results on the Dirichlet problem on hyperbolic spaces. Consider the eigenvalue problem in the hyperbolic $n$-space $\mathbb{H}^{n}(-k)$ with constant sectional curvature $-k<0$ given by
\begin{eqnarray}  \label{2.1}
\left\{
\begin{array}{ll}
\Delta{v}+\lambda{v}=0  \quad&\mathrm{in}\quad ~B_{\mathbb{H}^{n}(-k)}(R) ,\\
v=0  \quad&\mathrm{on}\quad
\partial{B}_{\mathbb{H}^{n}(-k)}(R) ,
\end{array}
\right.
\end{eqnarray} where $B_{\mathbb{H}^{n}(-k)}(R)$ is a geodesic ball of radius $R >0$ in $\mathbb{H}^{n}(-k)$. One does not need to specify a center for the geodesic ball since the hyperbolic space is two-points homogeneous, which implies that the first Dirichlet eigenvalue of the Laplacian of two geodesic balls of same radius but different centers are the same.

On the one hand, consider geodesic polar coordinates $(t,\xi)\in[0,+\infty)\times\mathbb{S}^{n-1}$ set up at arbitrary point $p$ of $\mathbb{H}^{n}(-k)$, the Laplace operator $\Delta$ can
be rewritten as
\begin{eqnarray*}
\Delta=\frac{d^{2}}{dt^{2}}+(n-1)\sqrt{k}\coth(\sqrt{k}t)\frac{d}{dt}+\left(\frac{\sqrt{k}}{\sinh(\sqrt{k}t)}\right)^{2}\cdot\Delta_{\mathbb{S}^{n-1}},
\end{eqnarray*}
where $t=d(p,\cdot)$ is the distance to $p$ on $\mathbb{H}^{n}(-k)$
and $\Delta_{\mathbb{S}^{n-1}}$ is the Laplacian on the unit
$(n-1)$-sphere $\mathbb{S}^{n-1}$. By Courant's nodal domain Theorem
(see, e.g., page 19 of \cite{i}), we know that the dimension of the
eigenspace of the first Dirichlet eigenvalue is $1$ and its
eigenfunction is the only eigenfunction which \emph{cannot change
sign} within the specified domain. Based on these two facts and
(\ref{2.1}), we know that the first Dirichlet eigenvalue
$\lambda_{1}(R)$ of \emph{the Laplacian on a geodesic ball of radius
$R$ in $\mathbb{H}^{n}(-k)$} and its eigenfunction $v$ satisfies the
following ODE
\begin{eqnarray}\label{Eq:Radial}
\left\{
\begin{array}{lll}
\frac{d^{2}v}{dt^{2}}+(n-1)\sqrt{k}\coth(\sqrt{k}t)\cdot\frac{dv}{dt}+\lambda_{1}(R)\cdot{v}=0,\\
\\
\frac{dv}{dt}(0)=v(R)=0,
\end{array}
\right.
\end{eqnarray}
which implies that the corresponding eigenfunction $v$ is radial.
There are several interesting estimates for the first Dirichlet eigenvalue $\lambda_{1}(R)$ in $\mathbb{H}^{n}(-k)$ that we would like to mention here.
More precisely, McKean \cite{mc} proved that $\lambda_{1}(R)$
satisfies
\begin{eqnarray*}
\lambda_{1}(R)\geqslant\frac{(n-1)^{2}k}{4} \text{ for all } R>0,
\end{eqnarray*}
and moreover, the asymptotical property
$$\lim\limits_{R\rightarrow+\infty}\lambda_{1}(R)= \frac{(n-1)^2k}{4}$$ holds. Savo \cite{sa} improved McKean's result in the following sense: if
$k=1$, he gave the estimate
 \begin{eqnarray*}
 \frac{(n-1)^{2}}{4}+\frac{\pi}{R^{2}}-\frac{4\pi^{2}}{(n-1)R^{3}}\leqslant\lambda_{1}(R)\leqslant
 \frac{(n-1)^{2}}{4}+\frac{\pi}{R^{2}}+\frac{c}{R^{3}},
\end{eqnarray*}
where $c=\frac{\pi^{2}(n-1)(n+1)}{2}\int\limits_{0}^{\infty}\frac{t^2}{\sinh^{2}t}dt$. Moreover, this estimate can be sharpen if $n=3$. More precisely, if $n=3$, Savo proved that the first Dirichlet eigenvalue $\lambda_{1}(R)$ in $\mathbb{H}^{n}(-k)$ is $\lambda_{1}(R)=k+\frac{\pi^{2}}{R^{2}}$. Recently, Savo's estimates has been generalized by Artamoshin. In fact, Artamoshin \cite{as} gave estimates for the first Dirichlet eigenvalue $\lambda_{1}(R)$ in $\mathbb{H}^{n}(-k)$ as follows: $\frac{k}{4}+\left(\frac{\pi} {2R}\right)^{2}\leqslant\lambda_{1}(R)\leqslant\frac{k}{4}+\left(\frac{\pi}{R}\right)^{2}$ for $n=2$; he can obtain the same estimate as Savo's showed  using a different way for $n=3$; $\lambda_{1}(R)>\frac{(n-1)^{2}k} {4}+\left(\frac{\pi}{R}\right)^{2}$ for $n\geqslant4$. Therefore, according to the facts above and applying the domain monotonicity of eigenvalues (see, e.g., page 17 of \cite{i}),  we know that: \emph{for any number $\frac{(n-1)^{2}k}{4}<\lambda <+\infty$, there exists $R_{\lambda ,n}>0$ such that $\lambda_{1}(R_{\lambda,n})=\lambda$}. In other words, for any constant constant $\lambda>\frac{(n-1)^{2}k}{4}$ there exists a function $v$ such that
\begin{eqnarray} \label{2.2}
\left\{
\begin{array}{lll}
\Delta{v}+\lambda{v}=0  \quad&\mathrm{in}\quad ~B_{\mathbb{H}^{n}(-k)}(p,R_{\lambda,n}),\\
v>0  \quad&\mathrm{in}\quad ~B_{\mathbb{H}^{n}(-k)}(p,R_{\lambda,n}),\\
v=0  \quad&\mathrm{on}\quad
\partial{B}_{\mathbb{H}^{n}(-k)}(p,R_{\lambda,n}),
\end{array}
\right.
\end{eqnarray}
hods on a geodesic ball
$B_{\mathbb{H}^{n}(-k)}(p,R_{\lambda,n})\subset\mathbb{H}^{n}(-k)$,
with center $p\in\mathbb{H}^{n}(-k)$ and radius $R_{\lambda,n}$, and
$\lambda_{1}(R_{\lambda,n})=\lambda$. Clearly, this radius
$R_{\lambda,n}$ depends on $n$ and the chosen number $\lambda$, and
which can always be found.

Now, by (\ref{2.2}) and the maximum principle, we can prove the following \emph{narrow} property for the $f$-extremal domain on $\mathbb{H}^{n}(-k)$.

\begin{lemma}  \label{proposition1}
Assume that $\Omega$ is an open (bounded or unbounded) connected
domain of $\mathbb{H}^{n}(-k)$ ($n\geqslant2$) such that one can
find a (strictly) positive function $u\in{C}^{2}(\Omega)$ that
solves the equation
\begin{eqnarray} \label{2.3}
\Delta{u}+f(u)=0  \quad\mathrm{in}\quad\Omega,
\end{eqnarray}
where $f:(0,+\infty)\rightarrow\mathbb{R}$ satisfies the property

\begin{quote}
{\bf $\mathbb{P}_{1}:$} There exists some positive constant $\lambda>\frac{(n-1)^{2}k}{4}$ such that $f(t)\geqslant\lambda{t}$ for all $t>0$.
\end{quote}

Then, $\Omega$ does not contain any closed geodesic ball of radius $R_{\lambda,n}$, where $R_{\lambda,n}$ is determined in (\ref{2.2}). Moreover, if $u$ satisfies the boundary conditions
\begin{eqnarray} \label{2.4}
u=0 \quad\mathrm{and}\quad
\langle\nabla{u},\vec{v}\rangle_{\mathbb{H}^{n}(-k)}=\alpha
\qquad\mathrm{on}~~\partial\Omega
\end{eqnarray}
for some negative constant $\alpha$, then either the closure
$\overline{\Omega}$ does not contain any closed geodesic ball of
radius $R_{\lambda,n}$ or $\Omega$ is a geodesic ball of radius
$R_{\lambda,n}$. Here, $\vec{v}$ is the outward unit vector along $\partial\Omega$.
\end{lemma}

\begin{proof}
In this proof, unless specified, $\mathbb{H}^{n}$ will denote $\mathbb{H}^{n}(-k)$. Let $u$ be a solution to (\ref{2.3}) with $f$ satisfying the property $\mathbb{P}_{1}$. Suppose that there exists a point $p\in\mathbb{H}^{n}$ such that $\overline{B(p,R_{\lambda,n})}\subseteq\Omega$.


Let $v$ be the solution to (\ref{2.2}) normalized to have $L^{2}$-norm $1$. Since $u>0$ in $\Omega$ and $v$ is bounded in $B(p,R_{\lambda,n})$, it is possible to find a constant $\epsilon>0$ such that the function
\begin{eqnarray*}
v_{\epsilon}:=\epsilon\cdot{v}
\end{eqnarray*}
satisfies the following properties:

(1) $v_{\epsilon}(x)\leqslant{u}(x)$ for any
$x\in\overline{B(p,R_{\lambda,n})}$;

(2) there exists some $x_{0}\in{B}(p,R_{\lambda,n})$ such that
$v_{\epsilon}(x_{0})=u(x_{0})$.
$\\$


Now, we would like to
apply the maximum principle to the function $u-v_{\epsilon}$. In
fact, by the property $\mathbb{P}_{1}$, we have
\begin{eqnarray*}
\Delta(u-v_{\epsilon})=-f(u)+\lambda{v_{\epsilon}}\leqslant-\lambda(u-v_{\epsilon})\leqslant0 \text{ in } \overline{B(p,R_{\lambda,n})},
\end{eqnarray*}
which implies that $u-v_{\epsilon}$ is a super-harmonic function on
$\overline{B(p,R_{\lambda,n})}$. Besides, we have
$(u-v_{\epsilon})(x)\geqslant0$ for any
$x\in\partial{B(p,R_{\lambda,n})}$. Hence, by applying the maximum
principle to $u-v_{\epsilon}$, we know that
$u-v_{\epsilon}$ should attain its minimum $0$ at the boundary
$\partial{B(p,R_{\lambda,n})}$. However, at the interior point
$x_{0}$ we also have $(u-v_{\epsilon})(x_{0})=0$. This is a
contradiction. Therefore, our assumption cannot hold, which means
that $\Omega$ does not contain any closed geodesic ball of radius
$R_{\lambda,n}$. This completes the proof of the first assertion.

We will prove the second assertion by contradiction. Assume that the
second claim is not true. Then there should exist some point
$p\in\mathbb{H}^{n}$ such that
$\overline{B(p,R_{\lambda,n})}\subseteq\overline{\Omega}$, and
moreover, $p$ can be chosen suitably such that the boundary of
$\overline{B(p,R_{\lambda,n})}$ \emph{internally} touches  the
boundary of $\Omega$ at some point $q$. The existence of the point
$q$ can always be assured. If at the beginning one chooses a point $p$ such that $\overline{B(p,R_{\lambda,n})}\cap\partial\Omega=\emptyset$, in this case, one just needs to move $\overline{B(p,R_{\lambda,n})}$ inside
$\Omega$ along a fixed direction gradually such that $\overline{B(p,R_{\lambda,n})}$ tangents internally to $\partial\Omega$ at some point, since $\partial \Omega$ is $C^2$, and then this point is just the point $q$ one wants to find. On the other hand, boundary conditions
(\ref{2.4}) imply that there exists a positive constant $\ell_{0}$
such that the function
\begin{eqnarray*}
v_{\ell_{0}}=\ell_{0}\cdot{v}
\end{eqnarray*}
has the following properties:

(1) $v_{\ell_{0}}(x)<u(x)$ for any $x\in{B(p,R_{\lambda,n})}$;

(2) the Neumann data of $v_{\ell_{0}}$ at the boundary
$\partial{B(p,R_{\lambda,n})}$ are equal to a constant $\beta$ such
that $\alpha<\beta<0$.
 $\\$ Defining a function $v_{\ell}$ as $v_{\ell}:=\ell\cdot{v}$ and
 then increasing the parameter $\ell$ starting from $\ell_{0}$
 gradually, one of the following two situations happens:

 (1) there exists some $x_{0}\in{B(p,R_{\lambda,n})}$ such that
 $v_{\ell}(x_{0})=u(x_{0})$, or

 (2) the Neumann data of $v_{\ell}$ becomes
 $\langle\nabla{v_{\ell}},\vec{v}\rangle_{\mathbb{H}^{n}}=\alpha$, and moreover
 $v_{\ell}(x)<u(x)$ for all $x\in{B(p,R_{\lambda,n})}$.
 $\\$

In case (1), applying the maximum principle to the function $u-v_{\ell}$, it follows that $u\equiv{v_{\ell}}$ and then $\Omega=B(p,R_{\lambda,n})$.

In case (2), we know that
\begin{eqnarray*}
\Delta(u-v_{\ell})=-f(u)+\lambda{v_{\ell}}\leqslant-\lambda(u-v_{\ell})\leqslant0
\, \text{ in } \, \overline{B(p,R_{\lambda,n})},
\end{eqnarray*}
which implies that $u-v_{\ell}$ is a super-harmonic function in
$B(p,R_{\lambda,n})$. Together with the fact that
$(u-v_{\ell})(q)=0$ and
$\langle\nabla(u-v_{\ell}),\vec{v}\rangle_{\mathbb{H}^{n}}=0$ at the point $q\in\partial\Omega\cap\partial{B(p,R_{\lambda,n})}$, we can obtain that $u-v_{\ell}$ vanishes in a neighborhood of $q$ within $\Omega$. This is contradict with the fact that $(u-v_{\ell})(x)>0$ for any interior point $x\in{B(p,R_{\lambda,n})}$. So, in case (2), $\Omega$ can only be a geodesic ball with radius $R_{\lambda,n}$. This completes the proof of the second assertion.
\end{proof}

The conclusion of Lemma \ref{proposition1} can be improved to
Hadamard manifolds (i.e., simply connected Riemannian manifolds with
non-positive sectional curvature) as follows.

\begin{lemma}  \label{proposition2}
Assume that $\Omega$ is an open (bounded or unbounded) connected
domain of an $n$-dimensional ($n\geqslant2$) Hadamard manifold $M$
whose sectional curvature $K$ of $M$ is pinched as follows
\begin{eqnarray*}
- k_{1}\leqslant{K}\leqslant{-k_{2}}\leqslant 0,
\end{eqnarray*}
with $k_{1}$ and $k_{2}$ two non-positive  constants. Assume that
one can find a (strictly) positive function $u\in{C}^{2}(\Omega)$
that solves the equation
\begin{eqnarray*}
\Delta{u}+f(u)=0  \quad\mathrm{in}\quad\Omega,
\end{eqnarray*}
where $f:(0,+\infty)\rightarrow\mathbb{R}$ satisfies the property $\mathbb{P}_{1}$ mentioned in Lemma \ref{proposition1} for some constant $\lambda$ satisfying
$\lambda>\frac{(n-1)^{2}k_{1}}{4}$.

Then, $\Omega$ does not contain any closed geodesic ball of radius
$\frac{c_{1}(n,k_{1})}{\sqrt{\lambda}}$, where $c_{1}(n,k_{1})$, only
depending on $n$ and $k_{1}$, is the first positive zero-point of the function $z(t)$ satisfying the following boundary value problem
\begin{eqnarray*}
\left\{
\begin{array}{lll}
z''(t)+(n-1)\sqrt{k_{1}}\coth(\sqrt{k_{1}}t)z'(t)+ z=0,\\
z'(0)=0,\\
 z(0)=1.
\end{array}
\right.
\end{eqnarray*}
Moreover, if $u$ satisfies the boundary conditions
\begin{eqnarray*}
u=0 \quad\mathrm{and}\quad
\langle\nabla{u},\vec{v}\rangle_{M^{n}}=\alpha
\qquad\mathrm{on}~~\partial\Omega
\end{eqnarray*}
for some negative constant $\alpha$, then either the closure $\overline{\Omega}$ does not contain any closed geodesic ball of radius $\frac{c_{1}(n,k_{1})}{\sqrt{\lambda}}$ or $\Omega$ is isometric to a geodesic ball of radius $\frac{c_{1}(n,k_{1})}{\sqrt{\lambda}}$ in $\mathbb H ^n (- k_1)$ and $u$ is given by \eqref{Eq:Radial}.
\end{lemma}

\begin{proof}
For any point $p\in{M}$ and a positive constant $\lambda >-\frac{(n-1)^{2}k_{1}}{4}$, there exists some constant $R_{\lambda,n,p}>0$, depending on $\lambda$, $n$, and the point $p$, such that $\lambda_{1}(B_{M^{n}}(p,R_{\lambda,n,p}))=\lambda$, where $B_{M}(p,R_{\lambda,n,p})$ is the geodesic ball on $M$ with center $p$ and radius $R_{\lambda,n,p}$, and, as before, $\lambda_{1}(\cdot)$ denotes the first Dirichlet eigenvalue of the Laplacian on the corresponding geodesic ball. So, there exists a function $v$ such that
\begin{eqnarray} \label{2.5}
\left\{
\begin{array}{lll}
\Delta{v}+\lambda{v}=0  \quad&\mathrm{in}\quad ~B_{M}(p,R_{\lambda,n,p}),\\
v>0  \quad&\mathrm{in}\quad ~B_{M}(p,R_{\lambda,n,p}),\\
v=0  \quad&\mathrm{on}\quad
\partial{B}_{M}(p,R_{\lambda,n,p}),
\end{array}
\right.
\end{eqnarray}hods. Clearly, $v$ is the eigenfunction of $\lambda_{1}(B_{M}(p,R_{\lambda,n,p}))=\lambda$. Now, based on $v$ which is determined by (\ref{2.5}), we can construct functions $v_{\epsilon}$ and $v_{\ell}$ as in the proof of Lemma \ref{proposition1} on the set $\overline{B_{M}(p,R_{\lambda,n,p})}$. Therefore, similar to the procedure in the proof of Lemma
\ref{proposition1}, by applying the maximum principe to the differences $u-v_{\epsilon}$ and $u-v_{\ell}$, where $u$ is the solution to $\Delta{u}+f(u)$=0, all the conclusions in Lemma \ref{proposition2} can be obtained except
\begin{eqnarray*}
R_{\lambda,n,p}\leqslant \frac{c_{1}(n,k_{1})}{\sqrt{\lambda}}
\end{eqnarray*}and the range for $\lambda$. Now we would like to prove these two remaining claims.

In fact, by Cheng's Eigenvalue Comparison Theorems (cf. \cite{ch1,ch2}), we have for $r_{0}>0$,
\begin{eqnarray} \label{2.6}
\lambda_{1}(V_{n}(k_{2},r_{0}))\leqslant\lambda_{1}(B_{M}(p,r_{0}))\leqslant\lambda_{1}(V_{n}(k_{1},r_{0}))
\end{eqnarray} holds, where $V_{n}(k_{i},r_{0})$ is the geodesic ball of radius $r_0$ in the space $n$-form of constant sectional curvature $k_{i}$ ($i=1,2$). We know that
$$\lambda_{1}(V_{n}(k_{i},r_{0}))\geqslant-\frac{(n-1)^{2}k_{i}}{4} \, \text{ and } \,
 \lim\limits_{r_{0}\rightarrow+\infty}\lambda_{1}(V_{n}(k_{i},r_{0}))=-\frac{(n-1)^{2}k_{i}}{4}. $$

Therefore, letting $r_{0}$ tends to infinity in (\ref{2.6}), one has
$$   \lambda_{1}(M):=\lim\limits_{r_{0}\rightarrow+\infty}\lambda_{1}(B_{M}(p,r_{0})) \leqslant -\frac{(n-1)^{2}k_{1}}{4}, $$and letting $r_{0}$ tends to zero one has
$$ \lim\limits_{r_{0}\rightarrow 0}\lambda_{1}(B_{M}(p,r_{0})) \geqslant \lim\limits_{r_{0}\rightarrow 0} \lambda_{1}(V_{n}(k_{2},r_{0})) = +\infty $$

If $\lambda>-\frac{(n-1)^{2}k_{2}}{4}$, by the \emph{domain monotonicity of eigenvalues}, we have that there exists $R_{1}$ such that
\begin{eqnarray*}
\lambda_{1}(B_{M}(p,R_{1}))\leqslant  \lambda_{1}(V_{n}(k_{1},R_{1}))=\lambda ,
\end{eqnarray*}and hence, by the  \emph{domain monotonicity of eigenvalues} again, for any $p \in M$ there exists  $0 < R_{\lambda,n,p} \leqslant R_{1}$ such that

$$\lambda_{1}(B_{M}(p,R_{\lambda,n,p}))= \lambda ,$$which implies the existence of the solution $v$ to (\ref{2.5}). Also, $R_{\lambda , n , p} = R_1$ if, and only if, $B_{M}(p,R_{1})$ is isometric to $V_{n}(k_{1},R_{1})$ by Cheng's Eigenvalue Comparison Theorem. Hence, the solution $u$ must be given by \eqref{Eq:Radial}.

Moreover, as mentioned before, when we focus on the first Dirichlet eigenvalue, the eigenvalue problem (\ref{2.1}) in the hyperbolic space can be degenerated to an ODE, and this fact is also valid for the Euclidean space and the sphere. Therefore, in the space forms, the first Dirichlet eigenvalue of the Laplacian on a geodesic ball can be computed exactly once the radius is prescribed. In fact, based on this truth, one can easily know that
$\lambda_{1}\left(\frac{c_{1}(n,k_{1})}{\sqrt{\lambda}}\right)=\lambda$
and
$\lambda_{1}\left(\frac{c_{2}(n,k_{2})}{\sqrt{\lambda}}\right)=\lambda$,
with $c_{i}(n,k_{i})$ ($i=1,2$) determined by ODEs of the forms as
the one above in Lemma \ref{proposition2}. Together with the
fact $R_{2}\leqslant{R_{\lambda,n,p}}\leqslant{R_{1}}$ shown above,
we have
 \begin{eqnarray*}
\frac{c_{2}(n,k_{2})}{\sqrt{\lambda}}\leqslant
R_{\lambda,n,p}\leqslant\frac{c_{1}(n,k_{1})}{\sqrt{\lambda}}.
 \end{eqnarray*}
 However, we claim that the radius $R_{\lambda,n,p}$ can be only
 chosen to be $\frac{c_{1}(n,k_{1})}{\sqrt{\lambda}}$. This is
 because, in the case of Hadamard manifolds, $R_{\lambda,n,p}$ also
 depends on the choice of $p$. Here we would like to explain this
 claim using a very interesting example. For instance, we can assume that $M$ is a Hadamard manifold having two subsets $\Gamma_{1}$, $\Gamma_{2}$ such that $M\setminus(\Gamma_{1}\cup\Gamma_{2})\neq{\O}$, $K|_{\Gamma_1}=- k_{1}$, $- k_1\leqslant K|_{M\setminus(\Gamma_{1}\cup\Gamma_{2})}\leqslant{- k_2}$, and $K|_{\Gamma_1}= - k_{2}$. If furthermore the $f$-extremal
 domain $\Omega$ intersects $\Gamma_1$, $\Gamma_2$, and
 $M\setminus(\Gamma_{1}\cup\Gamma_{2})$ simultaneously, then the
 suitable radius we can choose is only
 $\frac{c_{1}(n,k_{1})}{\sqrt{\lambda}}$. Our claim follows. This
 completes the proof of Lemma \ref{proposition2}.
\end{proof}

\begin{remark}
\rm{ In fact, Cheng's eigenvalue comparison theorems have been
improved to more generalized forms for complete manifolds with
\emph{radial} (Ricci or sectional) curvature bounded (cf.
\cite[Theorems 3.6 and 4.4]{fmi}). Even for the nonlinear
$p$-Laplacian
$\Delta_{p}(\cdot)=\mathrm{div}\left(|\nabla(\cdot)|^{p-2}\nabla(\cdot)\right)$
with $1<p<\infty$, which is a natural generalization of the linear
Laplace operator, a Cheng-type eigenvalue comparison result can also
be achieved for complete manifolds with \emph{radial} Ricci
curvature bounded from below (cf. \cite[Theorem 3.2]{m1}). }
\end{remark}

\subsection{Boundary at infinity of a $f$-extremal domain}

The aim now is to study the boundary at infinity of a $f$-extremal domain. However, in order to show the application clearly, we prefer to recall some preliminaries about Hadamard manifolds first. For more details, see for instance \cite{pe}.

Let $M$ be a simply connected Hadamard manifold. It is well known that the cut locus of any point
on $M$ is empty, which implies that for any two points on $M$, there is a unique geodesic joining them. Therefore, the concept of geodesic convexity can be naturally defined for sets on $M$.

Let $v_i$ ($i=1,2$) be two unit vectors in $TM$ and let $\gamma_{v_i}(t)$, $i=1,2$, be two unit-speed geodesics on $M$ satisfying $\gamma'_{v_i}(0)=v_i$. We say that two geodesics $\gamma_{v_1}(t)$ and $\gamma_{v_2}(t)$ are \emph{asymptotic} if there exists a constant $c$ such that the distance $d(\gamma_{v_1}(t),\gamma_{v_2}(t))$ is less than $c$ for all $t\geqslant0$. Similarly, two unit vectors $v_1$ and $v_2$ are asymptotic if the corresponding geodesics $\gamma_{v_1}(t)$, $\gamma_{v_2}(t)$ have this property. It is easy to find that being asymptotic is an equivalence relation on the set of unit-speed geodesics or on the set of unit vectors on $M$. Every element of these equivalence classes is called a point at infinity. Denote by $M_{\infty}$ the set of points at infinity, and denote by $\gamma(+\infty)$ or $v(\infty)$ the equivalence class of the corresponding geodesic $\gamma(t)$ or unit vector $v$.

Assume that the Hadamard manifold $M$ has a sectional curvature bounded from above by a negative constant. Then for two asymptotic geodesics $\gamma_1$ and $\gamma_2$, the distance between the two curves $\gamma_{1}|_{[t_0,+\infty)}$,  $\gamma_{2}|_{[t_0,+\infty)}$ is zero for any $t_0\in\mathbb{R}$. Besides, for any $x, y\in M_{\infty}$, there exists a unique oriented unit speed geodesic $\gamma(t)$ such that $\gamma(+\infty)=x$ and $\gamma(-\infty)=y$, with $\gamma(-\infty)=y$ the corresponding point at infinity when we change the orientation of $\gamma$.

For any point $p\in M$, there exists a bijective correspondence between a set of unit vectors at $p$ and $M_{\infty}$. In fact, for a point $p\in M$ and a point $x\in M_{\infty}$, there exists a unique oriented unit speed geodesic $\gamma$ such that $\gamma(0)=p$ and $\gamma(+\infty)=x$. Equivalently, the unit vector $v$ at the point $p$ is mapped to the point at infinity $v(\infty)$. Therefore, $M_{\infty}$ is bijective to a unit sphere.

Set $M^\ast=M\cup M_{\infty}$. For a point $p\in M$, let $\mathcal {U}$ be an open set in the unit sphere of the tangent space $T_{p}M$. For any $r>0$, define
 \begin{eqnarray*}
 T(\mathcal{U},r):=\{\gamma_{v}(t)\in
 M^{\ast}|v\in\mathcal{U},~r<t\leqslant+\infty\}.
\end{eqnarray*}
Then we can construct a unique topology $\mathscr{T}$ on $M^{\ast}$
as follows: the restriction of $\mathscr{T}$ to $M$,
$\mathscr{T}|_{M}$, is the topology induced by the Riemannian
distance; the sets $T(\mathcal{U},r)$ containing a point $x\in
M_{\infty}$ form a neighborhood basis at $x$. We call such topology
the \emph{cone topology} of $M^\ast$. Clearly, the cone topology
$\mathscr{T}$ satisfies the following properties:

(A1) $\mathscr{T}|_{M}$ coincides with the topology induced by the
Riemannian distance;

(A2) for any $p\in M$ and any homeomorphism
$h:[0,1]\rightarrow[0,+\infty]$, the function $\varphi$, from the
closed unit ball of $T_{p}M$ to $M^\ast$, given by
$\varphi(v)=\exp_{p}(h(\|v\|)v)$ is a homeomorphism. Moreover,
$\varphi$ identifies $M_{\infty}$ with the unit sphere;

(A3) for a point $p\in M$, the mapping $v\rightarrow v(\infty)$ is
a homeomorphism from the unit sphere of $T_{p}M$ onto $M_{\infty}$.

Using the notion of the cone topology one can define the boundary at
infinity of a subset of $M$. In fact, given a subset $A\subseteq M$,
its boundary at infinity is the set $\partial A\cap M_{\infty}$,
where $\partial A$ is the boundary of $A$ w.r.t. the cone topology.
Denote by $\partial_{\infty}A$ the boundary at infinity of $A$,
which implies $\partial_{\infty}A=\partial A\cap M_{\infty}$.

Now, based on the above brief introduction, we can define Busemann functions and then horospheres. Given an unit vector $v $ in $ TM$, let $\gamma_{v}(t)$ be the oriented geodesic on $M$ satisfying $\gamma'_{v}(0)=v$, then the Busemann function $B_{v}:M\rightarrow\mathbb{R}$, associate to $v$, is defined by
\begin{eqnarray*}
B_{v}(p)=\lim\limits_{t\rightarrow+\infty}d(p,\gamma_{v}(t))-t.
\end{eqnarray*}

It is not difficult to see that this function has the following properties (cf. \cite{pe}):

(B1)  $B_{v}$ is a $C^2$ convex function on $M$;

(B2)  the gradient $\nabla B_{v}(p)$ is the unique unit vector $w$
at $p$ such that $v(\infty)=-w(\infty)$;

(B3)  if $w$ is a unit vector such that $v(\infty)=w(\infty)$, then
$B_{v}-B_{w}$ is a constant function on $M$.

Given a unit vector $v$ in $TM$ which is mapped to a point at infinity, say $x$, clearly, $x\in M_{\infty}$. The horospheres based at $x$ are defined to be the \emph{level sets} of the Busemann function $B_{v} $. We denote by $H_{x}(t)$ the horosphere based at $x$ at distance $t$, that is,
$$ H_{x} (t)=\{ p \in M \, : \,\, B_{v}(p)=t ,\, v(+\infty) = x \} . $$

The horospheres at $x$ give a foliation of $M$, and by (B1), we know that each element of this foliation bounds a convex domain in $M$ which is called a horoball. By (B2), the intersection between a geodesic $\gamma$ and a horosphere at $\gamma(+\infty)$ is always orthogonal. By (B3), the horospheres at $x$ do not depend on the choice of $v$.

Denote by $\mathrm{int}(\cdot)$ the interior of a given set of points, we can obtain the following.

\begin{lemma}\label{horosphere}
Assume that $\Omega$ is an open (bounded or unbounded) connected domain of an $n$-dimensional ($n\geqslant2$) Hadamard manifold $M$ whose sectional curvature $K$ of $M$ is pinched as follows
\begin{eqnarray*}
-k_{1}\leqslant{K} \leqslant{-k_{2} }< 0,
\end{eqnarray*}
with $k_{1}$ and $k_{2}$ two positive  constants. Assume that one can find a (strictly) positive function
$u\in{C}^{2}(\Omega)$ that solves the equation
\begin{eqnarray*}
\Delta{u}+f(u)=0  \quad\mathrm{in}\quad\Omega,
\end{eqnarray*}
where $f:(0,+\infty)\rightarrow\mathbb{R}$ satisfies the property $\mathbb{P}_{1}$ mentioned in Lemma \ref{proposition1} for some constant $\lambda$ satisfying $\lambda>\frac{(n-1)^{2}k_{1}}{4}$. Then, $\mathrm{int}(\partial_{\infty}\Omega)={\O}$.
\end{lemma}

\begin{proof}
Assume that $\mathrm{int}(\partial_{\infty}\Omega)\neq{\O}$ and let $x\in\mathrm{int}(\partial_{\infty}\Omega)\subseteq M_{\infty}$ be an interior point. Consider the foliation by horospheres $H_{x}(t)$ based at $x$.

The sequence of horospheres $\{H_{x}(t)\}_{ t\in \mathbb R }$ converges to $x$ as $t\rightarrow+\infty$. Together with the fact $x \in \mathrm{int}(\partial_{\infty}\Omega)$, there exists some $T$ with $|T|<+\infty$ such that the horosphere $H_{x}(t)$ is completely contained in $\Omega\subseteq  M$ for all $t >T$.

Fix $t> T$. Let $\beta$ be the unique complete geodesic such that $\beta (+\infty)= x$ and $\beta (0) = p \in H _{x} (t)$. It is clear that $\beta(0,+\infty) \subset D_{x} (t)$, where $D_{x}(t)$ denotes the horoball bounded by $H_{x}(t)$, and $d(\beta (s) , H_{x}(t)) \to +\infty$ as $s \to +\infty$.

Thus, there exists $s_{0} >0 $ such that the geodesic ball centered
at $\beta(s_{0})$ of radius $\frac{c_{1}(n,k_1)}{\sqrt{\lambda}}$ is
completely contained in $D_{x}(t)\subset \Omega$, which 
contradicts the conclusion of Lemma \ref{proposition2}. Hence,
$\mathrm{int}(\partial_{\infty}\Omega)={\O}$.
\end{proof}

In fact, we can be more precise about the structure of the boundary at infinity of a $f$-extremal domain. But first, we shall need to introduce some notation. Given $x \in M_{ \infty }$, we define the {\it cone at infinity based on $x$ } of parameters $y\in M_{\infty} \setminus \{x\}$, $r >0 $ and $s \in \mathbb R $ as the set of points
\begin{equation}
\mathcal{C}_{x}(y,r,s)=\{ p \in M \, ; \,\, d(\gamma (\tilde s),p ) \leq r \text{ for all } \tilde s \geq s  \},
\end{equation}where $\gamma $ is the unique complete geodesic joining $x$ and $y$, that is, $\gamma (+\infty) = x $ and $\gamma (-\infty) = y $. With this, we define:

\begin{defn}\label{DefConical}
Let $\Omega \subset M$ be a connected domain such that $\partial _{\infty} \Omega \neq \O $. We say that $x \in \partial_{\infty} \Omega $ is a {\bf conical point of radius $r$} if there exist $y_{0} \in M_{\infty} $ and $s_{0} \in \mathbb R  $ such that $\mathcal C _{x } (y_{0},r ,s_{0}) \subset \Omega$.

Moreover, we say that {\bf $x$ is a horospherical point} if there exists $t \in \mathbb R$ such that $D_{x} (t) \subset \Omega$, here $D_{x}(t)$ is the horoball bounded by the horosphere $H_{x}(t)$.
\end{defn}

Note that a horospherical point is nothing but a conical point of radius {\it infinity}. Hence, now we can state:

\begin{theorem}\label{ThBoundary}
Let $\Omega$ be an open (bounded or unbounded) connected domain of an $n$-dimensional ($n\geqslant2$) Hadamard manifold $M$ whose sectional curvature $K$ of $M$ is pinched as follows
\begin{eqnarray*}
-k_{1}\leqslant{K} \leqslant{- k_{2} }< 0,
\end{eqnarray*}
with $k_{1}$ and $k_{2}$ two positive constants. Assume that one can find a (strictly) positive function
$u\in{C}^{2}(\Omega)$ that solves the equation
\begin{eqnarray*}
\Delta{u}+f(u)=0  \quad\mathrm{in}\quad\Omega,
\end{eqnarray*}
where $f:(0,+\infty)\rightarrow\mathbb{R}$ satisfies the property
$\mathbb{P}_{1}$ mentioned in Lemma \ref{proposition1} for
some constant $\lambda$ satisfying $\lambda>\frac{(n-1)^{2}k_{1}}{4}$.

Then, there is no conical point $x \in \partial _{\infty} \Omega$ of radius $r > \frac{c_{1}(n,k_1)}{\sqrt{\lambda}}$. In particular, the Hausdorff dimension of $\partial _{\infty} \Omega$ satisfies ${\rm dim}_{\mathcal H}\left( \partial _{\infty} \Omega \right) < n-1  $.
\end{theorem}
\begin{proof}
By contradiction, suppose there exists $x \in \partial _{\infty} \Omega$ a conical point of radius $r > \frac{c_{1}(n,k_1)}{\sqrt{\lambda}}$. Then, there exist $y_{0} \in  M_{\infty} \setminus \{x\} $ and $s_{0} \in \mathbb R$ such that $\mathcal C _{x}(y_{0},r ,s_{0})\subseteq \Omega$.

Then, for all $s > s_{0 } + \frac{c_{1}(n,k_1)}{\sqrt{\lambda}}$ the ball centered at $\beta (s)$ of radius $\frac{c_{1}(n,k_1)}{\sqrt{\lambda}}$ is contained in $\mathcal C _{x } (y_{0},r,s_{0})$, which contradicts Lemma \ref{proposition2}.

Now, if the Hausdorff dimension $\partial _{\infty} \Omega$ were $n-1$ then $\partial _{\infty} \Omega$ would contain an open set, and therefore $\partial _{\infty} \Omega$ would contain a horospherical point. This contradicts Lemma \ref{horosphere}.
\end{proof}

As an immediate consequence we get

\begin{corollary}\label{CorBoundary}
If $f$ satisfies property $\mathbb{P}_{1}$ mentioned in Lemma \ref{proposition1} for some constant $\lambda$ satisfying $\lambda>\frac{(n-1)^{2}k_{1}}{4}$, then a horoball can not be a $f$-extremal domain in a Hadamard manifold $M$ of sectional curvature $-k_{1}\leqslant{K} \leqslant{-k_{2} }< 0$.
\end{corollary}

\subsection{Concluding remarks}

We would like to close this section by making some analogies of these overdetermined elliptic problems, CMC hypersurfaces in the hyperbolic space and the singular Yamabe Problem.

The geometric idea behind Theorem \ref{ThBoundary} is that {\it the mean convex side of a properly embedded CMC $H$ hypersurface $\Sigma \subset \mathbb H ^{n}$ cannot contain a sphere of the same mean curvature}. Hence, in particular, a properly embedded CMC $H>1$ hypersurface in $\mathbb H ^{n}$ cannot contain a horospherical point at its boundary at infinity.

Also, from the works of Mazeo-Pacard \cite{MP,MP2}, Espinar-G\'{a}lvez-Mira \cite{EGM} and Bonini-Espinar-Qing \cite{BEQ}, there exists a close relation between complete conformal metrics on subdomains of the sphere of constant positive scalar curvature (singular Yamabe Problem) and CMC-type hypersurfaces in the hyperbolic space. The singular Yamabe Problem is the following:

\begin{quote}
{\it Given a closed set $\Lambda \subset \mathbb S ^{n-1}$, $n\geqslant3$, called the singular set, does there exists a complete conformal metric to the standard metric of the sphere on $\mathbb S ^{n-1}\setminus \Lambda$ of constant positive scalar curvature?}
\end{quote}

One interesting result about the singular Yamabe problem is the following
\begin{quote}
{\bf Schoen-Yau Theorem \cite{SY}: } {\it The singular set $\Lambda \subset \mathbb S ^{n-1}$, $n\geqslant3 $, has Hausdorff dimension less or equals than $\frac{n-3}{2}$.}
\end{quote}

Hence, Theorem \ref{ThBoundary} and Schoen-Yau Theorem motivate us to conjecture:
\begin{quote}
{\bf Conjecture A: } {\it Let $\Omega$ be an open (bounded or unbounded) connected smooth domain of an $n$-dimensional ($n\geqslant3$) Hadamard manifold $M$ whose sectional curvature $K$ of $M$ is pinched as follows
\begin{eqnarray*}
-k_{1}\leqslant{K} \leqslant{-k_{2} }< 0,
\end{eqnarray*}
with $k_{1}$ and $k_{2}$ two positive constants. Assume that one can find a (strictly) positive function
$u\in{C}^{2}(\Omega)$ that solves the equation
\begin{eqnarray*}
\Delta{u}+f(u)=0  \quad\mathrm{in}\quad\Omega,
\end{eqnarray*}
where $f:(0,+\infty)\rightarrow\mathbb{R}$ satisfies the property
$\mathbb{P}_{1}$ mentioned in Lemma \ref{proposition1} for
some constant $\lambda$ satisfying
$\lambda>-\frac{(n-1)^{2}k_{1}}{4}$.

Then, the Hausdorff dimension of $\partial _{\infty} \Omega$ must be less or equal than $\frac{n-3}{2}$.}
\end{quote}

In dimension $n=2$, it must be possible to construct solutions to \eqref{1.3} in the set of points to a fixed distance from a complete geodesic in $\mathbb H ^{2}$. This set has two points at infinity. As far as we know, these examples are not explicitly known, nevertheless we think that following the works of P. Sicbaldi \cite{DeSi,sp} it would be possible to construct them.

A. Ros and P. Sicbaldi \cite{rs} proved narrow properties for $f$-extremal domains in the Euclidean Space based on geometric ideas developed in \cite{egr} for CMC surfaces. We are able to extend these geometric ideas to the context of OEP in Hadamard manifolds. Moreover, the hyperbolic structure of a Hadamard manifold will give information about the boundary at infinity of the $f$-extremal domain. As far as we know, there is no counterpart for this fact on CMC hypersurfaces properly embedded in a Hadamard manifold. That is, the equivalent to Conjecture A for CMC hypersurfaces would be

\begin{quote}
{\bf Conjecture B:} {\it  Let $M^n $, $n\geqslant3$, be a simply-connected Hadamard manifold whose sectional curvature $K$ is pinched as $ -k_{1}\leqslant{K} \leqslant{-k_{2} }< 0$, with $k_{1}$ and $k_{2}$ two positive constants. Let $\Sigma \subset M$ be a properly embedded CMC hypersurface whose mean curvature satisfies $H\equiv C$, where $C$ is a (big positive) constant depending on $k_{1}$, $k_{2}$ and $n$. Let $\Omega $ denote the mean convex side of $\Sigma$ in $M$.

Then, the Hausdorff dimension of $\partial _{\infty} \Omega$ must be less or equal than $\frac{n-3}{2}$.}
\end{quote}

Also, a natural problem to be posed in analogy to the Euclidean case is:

\begin{quote}
{\bf Conjecture C: } {\it Is the Berestycki-Caffarelli-Nirenberg Conjeture true in $\mathbb H^{n}$? That is, if $f$ is a Lipschitz function on $\mathbb{R}_{+}$, and $\Omega$ is a smooth domain in $\mathbb{H}^{n}$ such that $\mathbb{H}^{n}\backslash\overline{\Omega}$ is connected, then the existence of a bounded solution to OEP (\ref{1.3}) implies that $\Omega$ is either a geodesic ball $B^n(R)$ of radius $R$, a half-space determined by either a totally geodesic hyperplane or a equidistant hypersurface to a totally geodesic hyperplane, a generalized cylinder $B^{k}(R)\times\mathbb{H}^{n-k}$, where $B^{k}(R)$ is a ball in
$\mathbb{H}^{k}$ of radius $R$, a periodic perturbation of a generalized cylinder $B^{k}(R)\times\mathbb{H}^{n-k}$, or the complement of one of them.}
\end{quote}

We suspect that Conjecture C is not true for dimensions $n\geqslant3$ without adding the periodic perturbations of a generalized cylinder. One could try to construct examples as P. Sicbaldi \cite{sp} did in the Euclidean case.

We will prove the BCN-conjecture in dimension $n=2$ under certain circumstances (see Theorem \ref{BCNConjecture}).


\section{Symmetry and boundedness properties for the $f$-extremal domain on hyperbolic spaces} \label{section3}
\renewcommand{\thesection}{\arabic{section}}
\renewcommand{\theequation}{\thesection.\arabic{equation}}
\setcounter{equation}{0} \setcounter{maintheorem}{0}

In this section, we would like to investigate some symmetry and boundedness properties of the (bounded or unbounded) domains $\Omega\subseteq\mathbb{H}^{n}(-k)$ on which the OEP (\ref{1.4}) has a solution $u\in{C}^{2}(\overline{\Omega})$ (i.e., $f$-extremal domain). In order to obtain those results, it is better to use the \emph{Poincar\'{e} disk model}. Here, we make an agreement that in the sequel, unless specified, $\mathbb{H}^{n}$ will stand for $\mathbb{H}^{n}(-1)$.

\subsection{An important conclusion}

In this subsection, we give an important result which is the cornerstone of the usage of the moving plane method in the next subsection.

It is clear the equation \eqref{1.4} is invariant under rotations and hyperbolic translations. Invariant means that, if $u$ is a solution to \eqref{1.4} in $\Omega$, and $\mathscr{I} : \mathbb H ^{n } \to \mathbb H ^{n}$ is a rotation or a hyperbolic translation, then $v (p) = u(\mathscr{I}(p))$ is a solution to \eqref{1.4} in $\tilde \Omega = \mathscr{I}^{{-1}}(\Omega)$.

However, in order to obtain symmetry conclusions on the $f$-extremal domain, we must verify that \eqref{1.4} is invariant under reflections of $\mathbb H^{n}$.

Let $P$ be a totally geodesic hyperplane of $\mathbb{H}^{n}$, Then,
$P$ divides $\mathbb{H}^{n}$ into two connected components $P^{+}$
and $P^{-}$, i.e., $\mathbb{H}^{n}\setminus P = P^{+}\cup P^{-}$.
Let $ \mathscr{R}_{P} : \mathbb H ^{n } \to \mathbb H^{n} $ be the
isometry such that $\mathscr{R}_{P} (P^{+}) = P^{-}$,
$\mathscr{R}_{P} (P^{-}) = P^{+}$ and leaves invariant $P$,
$\mathscr{R}_{P} (P) = P$. That is, $\mathscr{R}_{P}$ is the {\bf
reflection through $P$}.

Let $\Omega$ be a (bounded or unbounded) connected  and
$\mathscr{R}_{P}$ be the reflection through $P$ on $\mathbb{H}^{n}$.
We denote by $\Omega_{-}$ the component $\Omega \cap P^{-}$, that we
assume to be nonempty, and denote by $\Omega_{+}$ its reflection
through $P$, i.e., $ \Omega _{+} = \mathscr{R}_{P} ( \Omega _{-}
)$. Define a function $w(p)$ as follows
\begin{eqnarray}  \label{ARF1}
w(p)=u(\mathscr{R}(p)) \text{ for } p \in \Omega _{+}  .
\end{eqnarray}

For the function $w$, we can prove the following.

\begin{lemma} \label{lemmaARF}
The function $w(p)$ defined by (\ref{ARF1}) satisfies the first PDE in the OEP (\ref{1.4}).
\end{lemma}
\begin{proof}
Here, in order to simplify computations, we use the \emph{upper
half-space model} of $\mathbb{H}^{n}$. By the upper half-space
model, $\mathbb{H}^{n}$ can be identified with the upper half-space
\begin{eqnarray*}
\left\{(y_{1},y_{2},\ldots,y_{n-1},y_{n})\in\mathbb{R}^{n}|(y_{1},y_{2},\ldots,y_{n})\in\mathbb{R}^{n-1},y_{n}\in\mathbb{R},y_{n}>0\right\}
\end{eqnarray*}
equipped with the metric
$\widetilde{g}_{-1}=\frac{dy_{1}^{2}+dy_{2}^{2}+\cdots+dy_{n-1}^{2}+dy_{n}^{2}}{y_{n}^{2}}$.
In this model, $\mathscr{R}_{P}$ is given as follows
\begin{eqnarray} \label{ALR}
\mathscr{R}_{P}:(Y,y_{n})\longrightarrow(Y_{0},0)+\frac{t^{2}(Y-Y_{0},y_{n})}{|Y-Y_{0}|^{2}+y_{n}^{2}},
\end{eqnarray}
with $t\in\mathbb{R}$ and $Y:=(y_{1},y_{2},\ldots,y_{n-1})$, which
are Euclidean inversions of the upper half-space w.r.t. the
Euclidean half-sphere with center $(Y_{0},0)$ and radius $t$. Recall
that Euclidean half-spheres centered at $(Y_{0},0)$ and radius $t$
are totally geodesic hyperplanes in the upper half-space model of
$\mathbb H^{n}$.

Without loss of generality, since \eqref{1.4} is invariant under rotations and hyperbolic translations, we can choose $Y_{0}=(0,0,\ldots,0)$ and $t_0=1$ here. Then for any point $p=(y_{1},y_{2},\ldots,y_{n-1},y_{n})=(Y,y_{n})$, we know that $\mathscr{R}$ is given by
\begin{eqnarray} \label{ALR2}
\mathscr{R}(p)=\frac{p}{|p|^2},
\end{eqnarray}
with $|\cdot|$ the Euclidean norm of $\mathbb{R}^{n}$.

Now, in order to show that $w(p)$ verifies the first PDE in (\ref{1.4}) at $p$, we need to calculate its Hessian. Set $e_{i}=(0,\ldots,0,1,0,\ldots,0)\in T_{p}\mathbb{H}^{n}=\mathbb{R}^{n}$, whose $i$-th ($1\leq i\leq n$) element is $1$ while the others are $0$. It is easy to check that $\widetilde{g}_{-1}(e_i,e_j)=\delta_{ij}/y_{n}^{2}$, $1\leq i,j\leq n$, where $\delta_{ij}$ are the Kronecker symbols. This implies that $\{e_1,e_2,\cdots,e_n\}$ is an orthogonal basis of $T_{p}\mathbb{H}^{n}$. By (\ref{ALR2}) we have

\begin{eqnarray*}
\mathscr{R}^{k}(p)=\langle\mathscr{R}(p),e_{k}\rangle_{\mathbb{R}^n}=\frac{\langle
p,e_{k}\rangle_{\mathbb{R}^n}}{|p|^2},
\end{eqnarray*}
and
\begin{eqnarray*}
 \nabla\mathscr{R}^{k}(p)=\frac{e_k}{|p|^2}-\frac{2\langle p,e_{k}\rangle_{\mathbb{R}^n}}{|p|^4}p=d\mathscr{R}_{p}(e_k),
 \end{eqnarray*}
 where $\langle \cdot,\cdot\rangle_{\mathbb{R}^n}$ is the standard Euclidean metric, and $\nabla$, $d$ are the gradient operator and the differential operator on $\mathbb{H}^n$, respectively. For convenience, in this proof, \emph{we would like to rewrite $\langle \cdot,\cdot\rangle_{\mathbb{R}^n}$ as $\langle \cdot,\cdot\rangle$}.

Let
\begin{eqnarray*}
 v_k:=\frac{e_k}{|p|^2}-\frac{2\langle p,e_{k}\rangle}{|p|^4}p,
\end{eqnarray*} for any $1\leq k\leq n$, and $p' = \mathscr{R}(p)$. Clearly, we have
\begin{eqnarray*}
 p'=\mathscr{R}(p)=\frac{p}{|p|^2}, \qquad \langle
 v_k,p\rangle=-\frac{\langle p,e_{k}\rangle}{|p|^2}.
\end{eqnarray*}

\begin{eqnarray} \label{add1}
 \frac{\partial^{2}w}{\partial x_{i}\partial
 x_{j}}(p)&=&\sum\limits_{k}\frac{\partial u}{\partial\mathscr{R}^k} \frac{\partial^{2}\mathscr{R}^k}{\partial x_{i}\partial
 x_{j}}+\sum\limits_{k,l}\frac{\partial^{2}u}{\partial\mathscr{R}^{k}\partial\mathscr{R}^{l}}\frac{\partial\mathscr{R}^{k}}{\partial x_{i}}\frac{\partial\mathscr{R}^{l}}{\partial
 x_{j}}\nonumber\\
 &=&\sum\limits_{k}\frac{\partial u}{\partial\mathscr{R}^k} \frac{\partial^{2}\mathscr{R}^k}{\partial x_{i}\partial
 x_{j}}+\mathrm{Hess}^{0}(u)_{p'}\left(d\mathscr{R}_{p}(e_i),d\mathscr{R}_{p}(e_j)\right),
 \end{eqnarray}
 where $\mathrm{Hess}^{0}$ is the Hessian w.r.t. the standard Euclidean metric of $\mathbb{R}^n$. On the other hand, by the definition of the Hessian operator,
 we have
 \begin{eqnarray}  \label{add2}
 \frac{\partial^{2}\mathscr{R}^k}{\partial x_{i}\partial
 x_{j}}(p)=\mathrm{Hess}^{0}(\mathscr{R}^k)_{p}(e_i,e_j)&=&-\frac{2}{|p|^4}\left(\langle p,e_i\rangle\delta_{jk}+\langle p,e_k\rangle\delta_{ij}
 +\langle p,e_j\rangle\delta_{ik}\right)+\nonumber\\
&&\qquad +\frac{8}{|p|^6}\langle p,e_i\rangle\langle
p,e_j\rangle\langle p,e_k\rangle
 \end{eqnarray}
 and
 \begin{eqnarray}  \label{add3}
 \frac{\partial
 u}{\partial\mathscr{R}^k}=du_{p'}(d\mathscr{R}_{p'}(e_k))=du_{p}(v_k).
 \end{eqnarray}
 Set $\widetilde{v}_k:=|p|^{2}v_{k}=e_{k}-\frac{2\langle
 p,e_{k}\rangle}{|p|^2}p$. Clearly, $\widetilde{v}_k=e_{k}+2\langle p,v_{k}\rangle \, p$, and $\{\widetilde{v}_{1},\widetilde{v}_{2},\ldots,\widetilde{v}_{n}\}$ is an orthonormal base at $p'$ since, for any $1\leq k,l\leq n$, we have
 \begin{eqnarray*}
 \langle\widetilde{v}_{k},\widetilde{v}_{l}\rangle=\langle
 e_{k},e_{l}\rangle-4\frac{\langle
 p,e_{k}\rangle\langle
 p,e_{l}\rangle}{|p|^2}+4\frac{\langle
 p,e_{k}\rangle\langle
 p,e_{l}\rangle}{|p|^4}|p|^2=\delta_{kl}.
 \end{eqnarray*}
 By (\ref{add2}) and (\ref{add3}), we have
 \begin{eqnarray} \label{add4}
 \sum\limits_{k}\frac{\partial^{2}\mathscr{R}^k}{\partial
 x_{i}^{2}}\frac{\partial u}{\partial\mathscr{R}^k}(p)&=&-\frac{2}{|p|^4}\sum\limits_{k}\left(2\langle p,e_i\rangle\delta_{ik}+\langle
 p,e_k\rangle\right)du_{p'}(v_k)+\frac{8}{|p|^6}\langle
 p,e_i\rangle^{2}\sum\limits_{k}\langle
 p,e_k\rangle du_{p'}(v_k) \nonumber\\
 &=&\frac{4}{|p|^2}\langle p,v_i\rangle
 du_{p'}(v_i)+\left(\frac{2}{|p|^6}-\frac{8\langle
 p,v_i\rangle^2}{|p|^4}\right)\alpha,
 \end{eqnarray}
where in the last equality of (\ref{add4}), we have set
$\alpha:=\sum_{k}\langle
 p,\widetilde{v}_k\rangle du_{p'}(\widetilde{v}_k)=\langle\nabla
 u(p'),p\rangle$. Combining (\ref{add1}) and (\ref{add4}), we have
 \begin{eqnarray}  \label{add5}
 \mathrm{Hess}^{0}(w)_{p}(e_i,e_i)=\mathrm{Hess}^{0}(u)_{p'}(v_i,v_i)+\frac{4}{|p|^2}\langle
 p,v_i\rangle du_{p'}(v_i)+\left(\frac{2}{|p|^6}-\frac{8\langle
 p,v_i\rangle^2}{|p|^4}\right)du_{p'}(p).
 \end{eqnarray}
 By the definition of $\widetilde{g}_{-1}$, we have
 $\widetilde{g}_{-1}=\frac{1}{y_{n}^{2}(p)}\langle\cdot,\cdot\rangle$ for any $p\in\mathbb{H}^n$, with $y_{n}(p)=\langle e_{n},p\rangle$.
 Set $\widetilde{g}_{-1}=e^{2\rho}\langle\cdot,\cdot\rangle$.
 Clearly, $\rho=-\log(y_{n})$. So, we have
 \begin{eqnarray*}
 \nabla^{0}\rho(p)=-\frac{e_{n}}{\langle
 p,e_n\rangle}=-\frac{e_n}{y_{n}(p)},
 \end{eqnarray*}
where $\nabla^0$ is the gradient w.r.t. the standard Euclidean metric of $\mathbb{R}^n$.

Note that the hyperbolic metric $\widetilde{g}_{-1}$ is \emph{conformal} to the standard Euclidean metric of $\mathbb{R}^n$ and hence, by direct computation, for any $f\in C^{2}(\mathbb{H}^n)$ and any $X, Y\in\mathscr{X}(\mathbb{H}^n)$, $\mathscr{X}(\mathbb{H}^n)$ the set of smooth vector fields on $\mathbb{H}^n$, we have
\begin{eqnarray*}
\mathrm{Hess}(f)(X,Y)=\mathrm{Hess}^{0}(f)(X,Y)+\langle X,Y\rangle\langle\nabla ^{0}{f},\nabla^{0}\rho\rangle-\langle X,\nabla^{0}\rho\rangle\langle\nabla^{0} f,Y\rangle-\langle\nabla^{0}\rho,Y\rangle\langle\nabla f,X\rangle.
\end{eqnarray*}

Therefore, by applying the above formula, we can directly obtain
\begin{eqnarray} \label{add6}
\mathrm{Hess}(w)_{p}(e_i,e_i)=\mathrm{Hess}^{0}(w)_{p}(e_i,e_i)-\frac{1}{\langle
p,e_n\rangle}\langle\nabla^{0} w(p),e_n\rangle+\frac{2\langle
e_i,e_n\rangle}{\langle p,e_n\rangle}\langle\nabla ^{0} w(p),e_i\rangle,
\end{eqnarray}

and
\begin{eqnarray} \label{add7}
\mathrm{Hess}
(u)_{p'}(v_i,v_i)=\mathrm{Hess}^{0}(u)_{p'}(v_i,v_i)-\frac{1}{|p|^{4}\langle
p,e_n\rangle}\langle\nabla^{0} u(p'),e_n\rangle+\frac{2\langle
v_i,e_n\rangle}{\langle p,e_n\rangle}\langle\nabla^{0} u(p'),v_i\rangle .
\end{eqnarray}
On the one hand, we have
\begin{eqnarray*}
dw_{p}(e_i)=\langle\nabla ^0 w(p),e_i\rangle=du_{p'}(v_i)=\langle\nabla ^0
 u(p'),v_i\rangle,
\end{eqnarray*}
and
\begin{eqnarray} \label{add8}
\nabla ^0 w(p)&=&\sum\limits_{i}dw_{p}(e_i)e_{i}=\sum\limits_{i}du_{p'}(v_i) e_{i}\nonumber\\
&=&\frac{1}{|p|^2}\sum\limits_{i}du_{p'}(\widetilde{v}_i)\left(\widetilde{v}_{i}-2\langle
p,v_{i}\rangle p\right)\nonumber\\
&=&\frac{1}{|p|^2}\nabla ^0
u(p')-\frac{2}{|p|^4}\left(\sum\limits_{i}du_{p'}(\widetilde{v}_{i})\langle
p,\widetilde{v}_{i}\rangle p\right)\nonumber\\
&=&\frac{1}{|p|^2}\nabla ^0 u(p')-\frac{2\langle\nabla ^0
u(p'),p\rangle}{|p|^4}p,
\end{eqnarray}
which implies that $|\nabla u|(p')=|\nabla w|(p)$ in the sense of
the hyperbolic metric $\widetilde{g}_{-1}$.

On the other hand, the gradient of a function $f$ by a conformal
change of metric is given by
\begin{eqnarray} \label{add9}
 \nabla f (p) = e^{2\rho(p)}\nabla ^{0}
f (p).
\end{eqnarray}

By Remark \ref{addr}, we know that the first PDE in (\ref{1.4}) can
be rewritten as $\mathcal{F}u=0$ with
$\mathcal{F}u=\mathrm{Tr}\left(A\nabla^{2}u\right)+f(u,|\nabla{u}|)$,
which is independent of the choice of local coordinates. If we
choose an orthogonal basis $\{e_1,\cdots,e_n\}$ at some point of
$\mathbb{H}^n$, then $\nabla^{2}u(p)$ can be diagonalized, which
implies that in this setting, we have
\begin{eqnarray*}
\mathcal{F}u(p')=\sum\limits_{i=1}^{n}a_{i}(u,|\nabla
u(p')|)\left(\widetilde{g}_{-1}(v_i,v_i)\right)^{-1}\mathrm{Hess}(u)_{p'}(v_i,v_i)+f(u,|\nabla
u(p')|).
\end{eqnarray*}
Similarly, we can get
\begin{eqnarray*}
\mathcal{F}w(p)=\sum\limits_{i=1}^{n}a_{i}(w,|\nabla
w(p)|)\left(\widetilde{g}_{-1}(e_i,e_i)\right)^{-1}\mathrm{Hess}(w)_{p}(e_i,e_i)+f(w,|\nabla
w(p)|).
\end{eqnarray*}
Substituting (\ref{add5})-(\ref{add9}) into the above two
equalities, we can get $\mathcal{F}u(p')=\mathcal{F}w(p)=0$ for any
$p\in\Omega_{+}$. This completes the proof of Lemma \ref{lemmaARF}.
\end{proof}

\begin{remark}
\rm{ Clearly, as a special case of the OEP (\ref{1.4}), the function $w(p)$ defined by (\ref{ARF1}) also satisfies the first PDE in the OEP (\ref{1.3}).}
\end{remark}

\subsection{Symmetry properties of the $f$-extremal domain}

Suppose now $\Omega$ is an open (bounded or unbounded) connected domain in $\mathbb{H}^{n}$ whose boundary is of class $C^{2}$ and on which there exists a solution $u\in{C}^{2}(\overline{\Omega})$ to the OEP (\ref{1.4}).

As we pointed out above, there exists a close relation between OEP and properly embedded CMC hypersurfaces in the hyperbolic space. A. Ros and P. Sicbaldi \cite{rs} showed this when the extremal domain is contained in the Euclidean Space $\Omega \subset \mathbb R ^{n}$. In this case, they showed analogous results to those for properly embedded CMC hypersurfaces in the Euclidean Space developed by Korevaar-Kusner-Meeks-Solomon \cite{KKMS,KKS,Me}.

In the Hyperbolic setting, our aim is to extend Levitt-Rosenberg Theorem \cite{LR} for OEP. In certain sense, the hyperbolic geometry imposes more restrictions to the extremal domain than the Euclidean geometry. Specifically:

\begin{theorem}\label{ThLR}
Assume that $\Omega$ is a connected open domain in $\mathbb{H}^{n}$, with properly embedded $C^{2}$ boundary $\Sigma$, on which the OEP (\ref{1.4}) has a solution $u\in{C}^{2}(\overline{\Omega})$.

Assume that $\partial _{\infty} \Omega \subset E $, where $E$ is an
equator at the boundary at infinity $\mathbb H^{n}_{\infty} =
\mathbb S^{n-1}$. Let $P$ be the unique totally geodesic hyperplane
whose boundary at infinity is $E$, i.e., $\partial _{\infty} P = E$.

It holds:
\begin{itemize}
\item If $\partial _{\infty} \Omega = \O$, then $\Omega$ is a geodesic ball and $u$ is radially symmetric.

\item If $\partial _{\infty} \Omega \neq \O$, then $\Omega$ is invariant by the reflection $\mathscr{R}_{P}$ through $P$ , i.e., $\mathscr{R}_{P} (\Omega) =\Omega$. Moreover, $u$ is invariant under $\mathcal R$, that is, $u(p) = u(\mathcal R (p))$ for all $p \in \Omega$.
\end{itemize}
\end{theorem}

We shall recall before we continue the relation between isometries
of the Hyperbolic Space $\mathbb H ^{n}$ and conformal
diffeomorphism on the sphere at infinity $\mathbb S ^{n-1}$. It is
well-known that an isometry in $\mathbb H ^{n}$ induces a unique
conformal diffeormorphism $\Phi$ in $\mathbb S ^{n-1}$ and
viceversa.

Hence, in the above Theorem \ref{ThLR} we only need to assume that $\partial _{\infty} \Omega \subseteq \partial B_{\mathbb S ^{n-1}}(x, r)$, where $B_{\mathbb S ^{n-1}}(x, r)$ is the geodesic ball in $\mathbb S ^{n-1}$ centered at $x $ of radius $r \in (0,\pi)$. In particular, an equator centered at $x$, $E(x)$, appears when $r = \pi /2$.

If $r \neq \pi /2 $, it is clear that there exists a unique
conformal diffeomorphism such that
$$\Phi (B_{\mathbb S ^{n-1}}(x, r)) = E(x) . $$

This conformal diffeomorphism corresponds to a hyperbolic translation that take $P_{r}$ into $P$. Here, $P_{r}$ and $P$ are the totally geodesic hyperplanes whose boundary at infinity are $\partial B_{\mathbb S ^{n-1}}(x, r)$ and $E(x)$ respectively. Since the OEP \eqref{1.4} is invariant under hyperbolic translations and rotations, then we only need to consider the equator centered at the north pole ${\bf n}\in\mathbb S^{n-1}$ in Theorem \ref{ThLR}.

So, from now on, let $E$ denote the equator centered at the north
pole ${\bf n} \in \mathbb S^{n-1}$ and let $P$ be the totally
geodesic hyperplane whose boundary at infinity is $E$. Let $\gamma :
\mathbb R \to \mathbb H^{n}$ be the complete geodesic (parametrized
by arc-length) joining the south and north poles, ${\bf s},{\bf n}
\in \mathbb S^{n-1}$. Let $P(t)$ be the totally geodesic hyperplane
orthogonal to $\gamma ' (t) $ at $\gamma (t) \in P(t)$ for all $t
\in \mathbb R$. It is clear that $\{P(t)\}_{t\in \mathbb R}$ defines
a foliation of $\mathbb H^{n}$ by totally geodesic hyperplanes such
that $P(0) = P$.

\vspace {3mm}

\begin{proof}[Proof of Theorem \ref{ThLR}.]

Let $\gamma $ be the complete geodesic joining the south and north
poles and $\{ P(t) \}_{ t \in \mathbb R}$ be the foliation by
totally geodesic hyperplanes orthogonal to $\gamma$ given above.

Since $\partial _{\infty}\Omega \subseteq E$, there exists $T <0$
such that $P(t) \cap \Omega = \O$ for all $t \geqslant T$. So, we
can increase $t$ up to the first contact point of $\partial \Omega$
and $P(t)$. Set $t _{1} \leqslant  0$ as this point.

We can assume $t_{1} <0$, otherwise we begin with the foliation coming from $+\infty$ and hence, we must find $t_{2} >0 $ such that $P(t) \cap \Omega = \O$ for all $t > t_{2} $ and $P(t_{2})$ has a first contact point with $\partial \Omega$. If $t_{2} $ were $0$, then $\Omega \subset P$, which is a contradiction. Hence, up to a rotation, we can assume $t_{1 } <0$.

Since $\partial _{\infty} P(t) \cap \partial _{\infty } \Omega = \O$
for all $t \neq 0$, we have that $ \Omega _{t}^{-} := P^{-}(t) \cap
\Omega $ is relatively compact in $M$ for all $t \in ( t_{1}, 0)$.
Here, $P^{-}(t)$ denotes the connected component of
$\mathbb{H}^{n}\setminus P(t)$ containing the south pole ${\bf s}$
on its boundary at infinity. Analogously, we define $\Omega
_{t}^{+}= P^{+}(t)\cap \Omega$.

For each $t \in (t_{1},0)$, set $\mathcal{R}_{t}$ the reflection through $P(t)$ and $\tilde \Omega _{t}^{+}:= \mathcal{R}_{t}(\Omega _{t} ^{-}) $. Since $\Sigma = \partial \Omega $ is $C^2$, there exists $\epsilon >0$ such that $ \tilde \Omega _t ^+ \subset \Omega _t ^+$ for all $t \in (t_1 , t_1 + \epsilon)$.

Now, for each $t \in (t_1 , t_1 +\epsilon)$ define a function $v_t (p)=u(\mathcal{R} _t(p))$, $p\in\tilde \Omega _t ^+$. By Lemma \ref{lemmaARF}, it follows that $v_t $ also satisfies the first PDE in the OEP (\ref{1.4}). So, we can obtain that the function $v_{t}$ satisfies
\begin{eqnarray*}
\left\{
\begin{array}{llll}
\sum\limits_{i=1}^{n}a_{i}(v_{t},|\nabla{v_{t}}|)\cdot{\partial^{2}_{ii}v_{t}}+f(v,|\nabla{v_{t}}|)=0  \quad&\mathrm{in}\quad ~~ \tilde \Omega _t ^+ ,\\
v_{t}(p)=u(p')  \quad&\mathrm{if}\quad ~~p'\in\partial\tilde\Omega _t ^+ \cap P(t),\\
v_{t}(p)=0  \quad&\mathrm{if}\quad ~~p\in\partial \tilde \Omega _t ^+ \cap\mathrm{comp}\left( P(t) )\right),\\
\langle\nabla{v_{t}},\vec{v}\rangle_{\mathbb{H}^{n}}=\alpha \quad&\mathrm{on}\quad
\partial\tilde \Omega _t ^+ \cap\mathrm{comp}\left(P(t)\right),
\end{array}
\right.
\end{eqnarray*} where $\mathrm{comp}\left(P(t)\right)$ is the complement set of the hyperplane $P(t)$ in $\mathbb H^n$. Here we would like to point out one thing, the Neumann data will not change by the reflection $\mathcal{R} _{t}$ through the hyperplane  $P(t)$, $\mathcal{R}_{t}$ inverts the gradient vector and the unit outward normal vector simultaneously. Since the gradient is constant along the normal direction $\langle\nabla{u},\vec{v}\rangle_{\mathbb{H}^{n}}=\alpha$, then $u - v_t >0$ in $\tilde \Omega _t ^+$ for all $t \in (t_1 , t_1 +\epsilon)$, shrinking $\epsilon$ if necessary.

Define the quase-linear elliptic operator $Q$ as
$$ Q h:=\sum\limits_{i=1}^{n}a_{i}(h,|\nabla{h}|)\cdot{\partial^{2}_{ii} h }+f(h,|\nabla{h}|) \, , \,\, h \in C^{2}(U), $$where $U$ is a relatively compact domain in $\mathbb H^{n}$.

Then, since $u$ and $v_{t}$ satisfy $ Q u = 0 = Q v_{t}$, it is easy to get (cf. \cite{gt}) that the function $w_{t} :=u-v_{t}$ satisfies a second order linear uniformly elliptic equation
\begin{eqnarray} \label{EqAux}
\left\{
\begin{array}{lll}
Q(w_{t})=0  \quad&\mathrm{in}\quad ~~~\tilde \Omega _t ^+ ,\\
w_{t} > 0  \quad&\mathrm{in}\quad ~~~\tilde \Omega _t ^+ ,\\
w_{t}=0  \quad&\mathrm{on}\quad ~~\partial\tilde \Omega _t ^+ \cap  P(t) ,\\
w_{t}\geqslant0  \quad&\mathrm{on}\quad
~~\partial\tilde \Omega _t ^+ \cap\mathrm{comp}\left(P(t)\right),
\end{array}
\right.
\end{eqnarray} where the last inequality in the above OEP holds since $u$ is positive in $\Omega$ and $v_{t}=0$ on $\partial\tilde \Omega _t ^+ \cap  P(t)$.

Now, we claim that:
\begin{itemize}
\item either $\tilde \Omega _{t}^{+} \subseteq \Omega _{t}^{+}$ and $w_{t} >0$ in $\tilde \Omega _{t}^{+} $ for all $t \in (t_{1 } ,0)$,
\item or, there exists $\bar t \in (t_{1},0)$ such that $P(\bar t)$ is a hyperplane of symmetry for $\Omega$, that is, $\mathcal{R} _{t}(\Omega)= \Omega $.
\end{itemize}

If this were not true, one of the following situations will happen:
\begin{itemize}
\item[(A)] There exists $\bar t \in (t_{1 } ,0)$ such that $\tilde \Omega _{\bar t}^{+} \subseteq \Omega _{\bar t}^{+}$ and $w_{\bar t}(q) = 0$ at some interior point $q \in \tilde \Omega _{\bar t}^{+} $.

\item[(B)]  There exists $\bar t \in (t_{1 } ,0)$ such that $\tilde \Omega _{\bar t}^{+}$ is internally tangent to the boundary of $\Omega _{\bar t}^{+} $ at some point not at $P(\bar t)$ and $\tilde \Omega _{ t}^{+} \subset \Omega _{ t}^{+}$ for all $t\in (0,\bar t)$.

\item[(C)] There exists $\bar t \in (t_{1 } ,0)$ such that $P(\bar t)$ arrives at a position where it is orthogonal to the boundary of $\Omega$ at some point.
\end{itemize}

If (A) happens, applying the strong maximum principle for linear elliptic operators to $w_{\bar t}$ yields $u-v_{\bar t}\equiv0  $ in $\tilde \Omega _{\bar t} ^+ $, which implies that $\tilde \Omega _{\bar t}^{+} \equiv \Omega _{\bar t}^{+}$. Therefore,
\begin{itemize}
\item either $P(\bar t)$ is a hyperplane of symmetry for $\Omega$, in which case $\partial _{\infty}\Omega =\O$ and $u(p) = u(\mathcal R _{\bar t} (p))$ for all $p\in \Omega$,

\item or $w_{t}>0$  in $\tilde \Omega _{t}^{+}$ as long as $\tilde \Omega _{t}^{+} \subseteq \Omega _{t}^{+}$.
\end{itemize}

Assume that (B) happens, that is, there exists $\bar t \in (t_{1 } ,0)$ such that $\tilde \Omega _{\bar t}^{+}$ is internally tangent to the boundary of $\Omega _{\bar t}^{+} $ at some point $p$ not at $P(\bar t)$ and $\tilde \Omega _{t}^{+} \subset \Omega _{ t}^{+}$ for all $t\in (0,\bar t)$. Clearly, we have $w_{\bar t} = u-v_{\bar t}=0$ at $p$. Together with $L(w_{\bar t})=0$ in  $\tilde \Omega _{\bar t}^{+}$, by applying the Hopf boundary maximum principle it follows that
\begin{eqnarray*}
\langle\nabla w_{\bar t},\vec{v}\rangle_{\mathbb{H}^{n}}>0 \qquad
\mathrm{at}~p.
\end{eqnarray*}

However, this contradicts the fact that
$\langle\nabla{u},\vec{v}\rangle_{\mathbb{H}^{n}}=\langle\nabla{v_{\bar t}},\vec{v}\rangle_{\mathbb{H}^{n}}=\alpha$. Therefore,
\begin{itemize}
\item either $P(\bar t)$ is a hyperplane of symmetry for $\Omega$, in which case $\partial _{\infty}\Omega =\O$  and $u(p) = u(\mathcal R _{\bar t} (p))$ for all $p\in \Omega$,

\item or $\tilde \Omega _{t}^{+}$ is never internally tangent to the boundary of $\Omega _{t}^{+} $ for all $t \in (t_{1 },0)$.
\end{itemize}

Assume (C) happens, that is, suppose that there exists $\bar t \in (t_{1 } ,0)$ such that $P(\bar t)$ arrives at a position where it is orthogonal to the boundary of $\Omega$ at some point $q$. In this
situation, even though we have $w_{\bar t} = u-v_{\bar t}=0$ at $q$, the boundary maximum
principle cannot be applied directly since $q$ is a right angled corner of $\tilde \Omega _{\bar t}^{+}$ and the requisite of the interior tangent ball is not available. We need to use a more delicate version of the boundary maximum principle to overcome this obstacle similar to what has been done by Serrin \cite{s}.

For this, we will show first that $w_{\bar t}$ has a zero of second order at $q$. In order to simplify the computation, we can use an isometry $\mathcal I$ of $\mathbb H^n$ to take the totally geodesic hyperplane $P(\bar t)$ to the equator passing through the origin given by $x_{1}=0$. Of course, the image of $q$ lies on this hyperplane. Furthermore, we can choose $\mathcal I$ suitably such that the inner normal at $q$ of the image of $\Omega$ lies along the $x_{n}$-axis. Hence, instead of introducing new notations for the images of domains under $\mathcal I$, we may assume that the totally geodesic hyperplane is given by $x_{1}=0$ and the inner normal to $\Omega$ at $q$ lies along the $x_{n}$-axis.

Since the boundary of $\Omega$ is of class $C^{2}$, in a sufficiently small neighborhood of $q$, the boundary of $\Omega$ can be seen as a graph over the coordinate hyperplane $x_{n}=0$, which implies that there exists a twice continuously differentiable function $\varphi$ such that in this small neighborhood, $\partial\Omega$ is represented by
\begin{eqnarray*}
x_{n}=\varphi(x_{1},x_{2},\ldots,x_{n-1}).
\end{eqnarray*}

So, near $q$, the Dirichlet condition $u=0$ can be rewritten as
\begin{eqnarray} \label{3.4}
 u(x_{1},x_{2},\ldots,x_{n-1},\varphi(x_{1},x_{2},\ldots,x_{n-1}))=0.
 \end{eqnarray}

From the local representation of $\partial\Omega$ near $q$, it is not difficult to construct a normal field, $\overrightarrow{N}$, to $\partial\Omega$ given by $\overrightarrow{N}=-\sum\limits_{i=1}^{n-1}\frac{\partial\varphi}{\partial x_{i}}\frac{\partial}{\partial x_{i}}+\frac{\partial}{\partial x_{n}}$. The orthogonality of $\overrightarrow{N}$ to the boundary
$\partial\Omega$ near $q$ can be checked easily since the hyperbolic metric $g_{-1}$ is conformally equivalent to the Euclidean metric.

Let
\begin{eqnarray*}
\rho(z):=\frac{4}{(1-|z|^{2})^{2}}=g_{-1}\left(\frac{\partial}{\partial
x_{i}}\Big{|}_{z},\frac{\partial}{\partial x_{i}}\Big{|}_{z}\right),
\qquad i=1,2,\ldots,n-1,
\end{eqnarray*} where $|z|$ is the Euclidean norm of a point $z$, and $g_{-1}$ is the hyperbolic metric. Normalizing $\overrightarrow{N}$ in the hyperbolic sense yields an inward unit normal field of $\partial\Omega$ as follows
\begin{eqnarray*}
\frac{\partial}{\partial\vec{n}}=\frac{1}{\sqrt{\rho}}\cdot\frac{1}{\sqrt{1+\sum\limits_{i=1}^{n-1}\left(\frac{\partial\varphi}{\partial
x_{i}}\right)^{2}}}\cdot\overrightarrow{N}.
\end{eqnarray*}

So, the Neumann-type condition $\langle\nabla{u},\vec{v}\rangle_{\mathbb{H}^{n}}=\alpha$ can be
rewritten as
\begin{eqnarray}  \label{3.5}
-\sum\limits_{i=1}^{n-1}\frac{\partial\varphi}{\partial
x_{i}}\frac{\partial{u}}{\partial x_{i}}+\frac{\partial{u}}{\partial
x_{n}}=-\alpha\cdot\sqrt{\rho}\cdot\sqrt{1+\sum\limits_{i=1}^{n-1}\left(\frac{\partial\varphi}{\partial
x_{i}}\right)^{2}}.
\end{eqnarray}

Differentiating (\ref{3.4}) w.r.t. the variable $x_i$, $1\leqslant i\leqslant n-1$, results into
\begin{eqnarray}  \label{3.6}
\frac{\partial{u}}{\partial x_{i}}+\frac{\partial{u}}{\partial
x_{n}}\cdot\frac{\partial\varphi}{\partial x_{i}}=0.
\end{eqnarray}

Evaluating (\ref{3.6}) at $q$, where $\frac{\partial\varphi}{\partial x_{i}}|_{q}=0$ for  $1\leqslant
i\leqslant n-1$, we have $\frac{\partial{u}}{\partial x_{i}}|_{q}=0$. Together with (\ref{3.5}), it follows that
\begin{eqnarray}  \label{3.7}
\frac{\partial{u}}{\partial
x_{n}}\Big{|}_{q}=-\alpha\cdot\left[\sqrt{1+\sum\limits_{i=1}^{n-1}\left(\frac{\partial\varphi}{\partial
x_{i}}\right)^{2}}\cdot\sqrt{\rho}\right]\Big{|}_{q}=-\alpha\sqrt{\rho}|_{q}.
\end{eqnarray}

Differentiating (\ref{3.6}) w.r.t. $x_i$ for $1\leqslant i\leqslant{n-1}$, evaluating at $q$ and using (\ref{3.7}), we have
\begin{eqnarray} \label{3.8}
\frac{\partial^{2}u}{\partial
x_{i}^{2}}\Big{|}_{q}=-\frac{\partial{u}}{\partial
x_{n}}\cdot\frac{\partial^{2}\varphi}{\partial
x_{i}^{2}}\Big{|}_{q}=\alpha\cdot\left[\sqrt{1+\sum\limits_{i=1}^{n-1}\left(\frac{\partial\varphi}{\partial
x_{i}}\right)^{2}}\cdot\sqrt{\rho}\right]\Big{|}_{q}\cdot\frac{\partial^{2}\varphi}{\partial
x_{i}^{2}}\Big{|}_{q}=\alpha\sqrt{\rho}|_{q}\cdot\frac{\partial^{2}\varphi}{\partial
x_{i}^{2}}\Big{|}_{q}.
\end{eqnarray}

Differentiating (\ref{3.5}) w.r.t. $x_{i}$ yields
\begin{eqnarray}  \label{3.9}
-\sum\limits_{i=1}^{n-1}\frac{\partial^{2}\varphi}{\partial
x_{i}^{2}}\frac{\partial{u}}{\partial
x_{i}}-\frac{\partial\varphi}{\partial
x_{i}}\frac{\partial^{2}u}{\partial x_{i}^{2}}
+\frac{\partial^{2}u}{\partial x_{n}\partial x_{i}}=-\alpha\cdot\frac{\partial\sqrt{\rho}}{\partial x_{i}}\cdot\sqrt{1+\sum\limits_{i=1}^{n-1}\left(\frac{\partial\varphi}{\partial x_{i}}\right)^{2}}-\nonumber\\
\alpha\cdot\sqrt{\rho}\cdot\frac{\partial}{\partial
x_{i}}\left(\sqrt{1+\sum\limits_{i=1}^{n-1}\left(\frac{\partial\varphi}{\partial
x_{i}}\right)^{2}}\right).
\end{eqnarray}

Note that for $1\leqslant i\leqslant n-1$, $\frac{\partial\varphi}{\partial x_{i}}|_{q}=0$, $\frac{\partial{u}}{\partial x_{i}}|_{q}=0$, and $\frac{\partial\sqrt{\rho}}{\partial x_{i}}|_{q}=0$ (this is because $\sqrt{\rho}$ is a radial function and $\frac{\partial}{\partial x_{i}}$ is tangent to a sphere centered at the origin $o$). So, evaluating (\ref{3.9}) at $q$, we can obtain
\begin{eqnarray}
\frac{\partial^{2}u}{\partial x_{n}\partial
x_{i}}\Big{|}_{q}=-\alpha\cdot\sqrt{\rho}\cdot\frac{\partial}{\partial
x_{i}}\left(\sqrt{1+\sum\limits_{i=1}^{n-1}\left(\frac{\partial\varphi}{\partial
x_{i}}\right)^{2}}\right)\Big{|}_{q}=0.
\end{eqnarray}

Applying the fact that $Qu=0$, and together with (\ref{3.8}), we can
evaluate $\frac{\partial^{2}u}{\partial\xi_{2}^{2}}$ at $q$ as
follows
\begin{eqnarray}
a_{n}|_{q}\cdot\frac{\partial^{2}u}{\partial
x_{n}^{2}}\Big{|}_{q}=-f|_{q}-\sum\limits_{i=1}^{n-1}\alpha\cdot{a}_{i}|_{q}\cdot\left[\sqrt{1+\sum\limits_{k=1}^{n-1}\left(\frac{\partial\varphi}{\partial
x_{k}}\right)^{2}}\cdot\sqrt{\rho}\right]\Big{|}_{q}\cdot\frac{\partial^{2}\varphi}{\partial
x_{i}^{2}}\Big{|}_{q}.
\end{eqnarray}

Now, we need to calculate the second-order partial derivatives of $v=u\circ\mathcal{R}$ at $q$. As we have mentioned above, through the suitable isometry on $\mathbb{H}^{n}$, the totally geodesic hyperplane $P(\bar t)$ can be given by $x_{n}=0$ and the inner normal vector of $\partial\Omega$ is along $x_{n}$-direction. Therefore, in this setting, the Alexandrov reflection $\mathcal{R}$ can be given simply as
$(x_{1},x_{2},\ldots,x_{n-1},x_{n})\rightarrow(x_{1},x_{2},\ldots,x_{n-1},-x_{n})$ along the $x_{n}-$axis,
and correspondingly, the function $v_{\bar t}$ can be expressed as follows
\begin{eqnarray*}
v_{\bar t}(x_{1},x_{2},\ldots,x_{n-1},x_{n})=u(x_{1},x_{2},\ldots,x_{n-1},-x_{n}).
\end{eqnarray*}
Therefore,  for $1\leqslant i\leqslant n-1$, we can get that
\begin{eqnarray*}
\frac{\partial{v_{\bar t}}}{\partial x_{i}}=\frac{\partial{u}}{\partial
x_{i}}=0, \qquad \frac{\partial{v_{\bar t}}}{\partial
x_{n}}=\frac{\partial{u}}{\partial x_{n}}, \qquad
\frac{\partial^{2}v_{\bar t}}{\partial x_{i}\partial
x_{n}}=-\frac{\partial^{2}u}{\partial x_{i}\partial x_{n}}=0, \qquad
\frac{\partial^{2}u}{\partial
x_{n}^{2}}=\frac{\partial^{2}v_{\bar t}}{\partial x_{n}^{2}}
\end{eqnarray*}
and
\begin{eqnarray*}
\frac{\partial^{2}v_{\bar t}}{\partial
x_{i}^{2}}=\frac{\partial^{2}u}{\partial
x_{i}^{2}}=\alpha\sqrt{\rho}|_{q}\cdot\frac{\partial^{2}\varphi}{\partial
x_{i}^{2}}\Big{|}_{q}=0.
\end{eqnarray*}

Here we would like to point one thing, that is, since in the situation (C), the reflected cap $\tilde \Omega _{\bar t}^{+}$  is contained in $\Omega _{\bar t}^{+}$ , the inner normal vector of $\partial\Omega$ at $q$ is along the $x_{n}$-axis, and the function $\varphi$ is twice continuously differentiable, one can get $\frac{\partial^{2}\varphi}{\partial x_{i}^{2}}|_{q}=0$ for $1\leqslant i\leqslant n-1$ by applying Taylor's theorem with remainder. So, we know that all the first-order and second-order partial derivatives of $u$ and $v_{\bar t}$ agree at $q$. Applying \cite[Lemma 2]{s} to $u-v_{\bar t}$, which is called \emph{the boundary point lemma at a corner} therein, we can obtain that either $\frac{\partial(u-v_{\bar t})}{\partial{s}}|_{q}>0$ or $\frac{\partial^{2}(u-v_{\bar t})}{\partial{s^2}}|_{q}>0$, where $\frac{\partial}{\partial{s}}$ denotes a constant vector field such
that $\frac{\partial}{\partial{s}}|_{q}$ enters $\Omega$ non-tangentially. Clearly, this is contradict with the fact that all the first-order and second-order partial derivatives of $u-v_{\bar t}$ vanish
at $q$.

Therefore,
\begin{itemize}
\item either $P(\bar t)$ is a hyperplane of symmetry for $\Omega$ and $u(p)=u(\mathcal R _{\bar t} (p))$ for all $p\in \Omega$, in which case $\partial _{\infty}\Omega =\O$,

\item or $P( t)$ never arrives at a position where it is orthogonal to the boundary of $\Omega$ at some point for all $t \in (t_{1 },0)$.
\end{itemize}

Summing up the above argument, we have shown that
\begin{itemize}
\item[(1)] either $\tilde \Omega _{t}^{+} \subseteq \Omega _{t}^{+}$ and $w_{t} >0$ in $\tilde \Omega _{t}^{+} $ for all $t \in (t_{1 } ,0)$,
\item[(2)] or, there exists $\bar t \in (t_{1},0)$ such that $P(\bar t)$ is a hyperplane of symmetry for $\Omega$, that is, $\mathcal{R} _{\bar t}(\Omega)= \Omega $  and $u(p)=u(\mathcal R _{\bar t} (p))$ for all $p\in \Omega$.
\end{itemize}

If (1) holds, the same must hold if we begin from $+\infty$, that is, $\tilde \Omega _{t}^{-} \subseteq \Omega _{t}^{-}$ and $w_{t} >0$ in $\tilde \Omega _{t}^{-} $ for all $t \in (0 , t_{2})$. But this implies that $P \equiv P(0)$ must be a hyperplane of symmetry  and $u(p)=u(\mathcal R (p))$ for all $p\in \Omega$.

If (2) holds, then $\partial _{\infty} \Omega = \O$ clearly and so
$\partial _{\infty} \Omega$ is included in all the equators of
$\mathbb S ^{n-1}$. Let $\mathbb{F}$ the set of all possible totally
geodesic hyperplanes $P$ about which $\Omega$ is symmetric. In the
group of M\"{o}bius transformations, let $\mathbb{G}$ be the closure
of the group generated by the reflections on $\mathbb{H}^{n}$ about
the hyperplanes $P$ in the family $\mathbb{F}$. So, $\mathbb{G}$ is
a compact group of isometries.

Using an argument involving center of mass (cf. \cite{kh}), we can get that $\mathbb{G}$ has a fixed point $m\in\mathbb{H}^{n}$. So, $\mathbb{F}$ consists of the set of all totally geodesic hyperplanes passing through $m$, and hence $\mathbb{G}$ contains the group of rotations about $m$. This implies that $\Omega$ is either a geodesic ball or a spherical shell. However, by characterization of each hyperplane in $\mathbb{F}$, we know that $\Omega$ cannot be a spherical shell. So, $\Omega$ must be a geodesic ball and $u$ is radially symmetric.

This finishes the proof.
\end{proof}

Hence, as a corollary we have

\begin{corollary} \label{theorem3.3-1}
Assume that $\Omega$ is a bounded connected open domain in $\mathbb{H}^{n}$, with $C^{2}$ boundary, on which the OEP (\ref{1.4}) has a solution $u\in{C}^{2}(\overline{\Omega})$. Then $\Omega$ must be a geodesic ball and $u$ is radially symmetric.
\end{corollary}

\begin{remark}
\rm{If the first equation in the OEP (\ref{1.4}) is \emph{simplified} to be $\Delta{u}=-1$ in $\Omega$, then the conclusion of Theorem \ref{theorem3.3-1} has been obtained by Molzon \cite{mr}. Equivalently, we have improved Molzon's conclusion to a more general situation.}
\end{remark}

Now, we would like to give another interesting application. However, before that we need the following so-called \emph{basic hyperbolic geometry} (cf. \cite{dCL}).

\begin{lemma}  \label{lemma3.5-1}
(Basic hyperbolic geometry) Let $\Sigma$ be a connected properly embedded hypersurface in hyperbolic $n$-space $\mathbb{H}^{n}$ whose asymptotic boundary consists of a single point $x\in \partial _{\infty } \mathbb H ^{n}$. Let $P$ be a totally geodesic hyperplane such that $x\in \partial_{\infty}P$. If $\Sigma$ is symmetric about every such totally geodesic hyperplane $P$, then $\Sigma$ is a horosphere. Furthermore, if $P_{\gamma}(t)$, $P_{\gamma}(t)\neq P$ is an arbitrary translated copy of $P$ along a geodesic $\gamma$ cutting orthogonally $P$, then $P_{\gamma}(t)\cap{\Sigma }$ is empty or else $P_{\gamma}(t)\cap{\Sigma}$ is compact.
\end{lemma}

In order to establish correctly the next result, we shall introduce some concepts on Hyperbolic Geometry. Given any point at infinity $x \in \partial _\infty \mathbb H ^n $, there exists a $(n-1)-$parameter family of parabolic translations $\{\mathcal T^x _{v} \}_{v\in \mathbb R^{n-1}}$ that fix $x$ at infinity and, hence,
$$  \mathcal T ^x_v (H_x (t)) = H_x (t) \text{ for all } v\in \mathbb R ^{n-1} \text{ and } t\in \mathbb R ,$$where $\{H_x (t) \}_{t\in \mathbb R}$ is the foliation by horospheres at $x \in \partial _\infty \mathbb H^n$.

Hence, one can check that given any $v \in \mathbb R ^{n-1}$ there exists two totally geodesic hyperplanes $P_1$ and $P_2$ such that $x \in \partial _\infty P_1\cap \partial _\infty P_2$ whose associated hyperbolic reflections $\mathcal R _1 ,\mathcal R _2 \in {\rm Iso}(\mathbb H ^n)$ satisfy
\begin{equation}\label{Eq:TransParabolic}
\mathcal T ^x_v = \mathcal R _1 \circ \mathcal R_2 .
\end{equation}

So, given a horoball $D_x (t)$, we can parametrize it as $(0,+\infty) \times \mathbb R ^{n-1}$ by
$$\begin{matrix}
(0,+\infty) \times \mathbb R ^{n+1} & \to & D_x(t) \\
(t,v) & \to & \mathcal T ^x _v (\gamma (t)) ,
\end{matrix}$$where $\gamma (t)$ is a geodesic with initial conditions $\gamma (0) \in H_x (t)$ and $\gamma ' (0)$ agrees with the inward normal of $H_x (t)$ at $\gamma (0)$.

\begin{defn}\label{Def:HoroSymmetric}
Given a $C^2$ function $u : D_x (t) \to \mathbb R$ is {\bf horospherically symmetric} if
$$ u(p) = u(\mathcal T^x _v (p)) \text{ for all } v \in \mathbb R^{n-1} .$$
\end{defn}

Applying the above lemma, we can prove the following.

\begin{theorem} \label{theorem3.6-1}
Assume that $\Omega$ is a domain in $\mathbb{H}^{n}$, with boundary a $C^2$ properly embedded hypersurface $\Sigma$ and whose asymptotic boundary is a point $x_{0} \in \partial _\infty \mathbb H ^n$, on which the OEP (\ref{1.4}) has a solution $u\in{C}^{2}(\overline{\Omega})$.

Then, $\Omega$ is a horoball $D_x (t)$, for some $t\in \mathbb R$ and $u$ is horospherically symmetric.
\end{theorem}
\begin{proof}
Since the boundary at infinity of $\Omega$ is a single point, we claim that Theorem \ref{ThLR} implies that $\Omega$ is symmetric with respect to every totally geodesic hyperplane containing $x_0\in \partial _\infty \Sigma$, that is, for any reflection $\mathcal R \in {\rm Iso}(\mathbb H ^n)$ that leaves invariant a totally geodesic hyperplane $P$ such that $x\in \partial _\infty P$, we have that $\mathcal R (\Omega) = \Omega$ and $u(p)=u(\mathcal R (p))$ for all $p\in \Omega$. Hence, Lemma \ref{lemma3.5-1} implies that $\Omega$ is a horoball and \eqref{Eq:TransParabolic} implies that $u$ is horospherically symmetric.

Let us prove the Claim. Let $B_{\mathbb S ^{n-1}}(x, r)$ be any geodesic ball in $\mathbb S^{n-1}$ that contains $x_0$ on its boundary, i.e., $x_0 \in \partial B_{\mathbb S ^{n-1}}(x, r)$. Let us denote by $P(x,r)$ the unique totally geodesic hyperplane with boundary at infinity $\partial _\infty P(x,r) = \partial B_{\mathbb S ^{n-1}}(x, r)$.

Then, there exists a unique isometry $\mathcal{I} _{x}$ that takes $\partial B_{\mathbb S ^{n-1}}(x, r)$ into an equator $E_x$ containing $x_0$. Hence, by Theorem \ref{ThLR}, the domain $\mathcal{I} _{x} (\Omega)$ is symmetry w.r.t. the totally geodesic hyperplane $P_x$ with boundary at infinity $\partial _\infty P _x = E_x$. Thus, if we undo the isometry $\mathcal{I} _{x}$, then $\Omega$ is symmetric w.r.t. the totally geodesic hyperplane $\mathcal{I} _{x}(P_x) = P(x,r)$, as claimed.
\end{proof}

This can be seen as the OEP version of the famous do Carmo-Lawson Theorem \cite{dCL}.

\begin{remark}
\rm{If the first equation in the OEP (\ref{1.4}) is \emph{simplified} to be $\Delta{u}=-1$ in $\Omega$, then the conclusion of Theorem \ref{theorem3.6-1} has been obtained by Sa Earp and Toubiana \cite{et}. Nevertheless, we have improved the conclusion of Sa Earp and Toubiana in \cite{et} to a more general situation.}
\end{remark}

In particular, Theorem \ref{theorem3.6-1} combined with Theorem \ref{ThBoundary} yields

\begin{theorem}\label{OneEnd}
There is no (strictly) positive function $u\in{C}^{2}(\Omega)$ that solves the equation
\begin{eqnarray*}
\Delta{u}+f(u)=0  \quad\mathrm{in}\quad\Omega,
\end{eqnarray*} where $f:(0,+\infty)\rightarrow\mathbb{R}$ satisfies the property $\mathbb{P}_{1}$ mentioned in Lemma \ref{proposition1} for some constant $\lambda$ satisfying $\lambda>\frac{(n-1)^{2}}{4}$, if $\Omega$ is a domain in $\mathbb{H}^{n}$  whose asymptotic boundary is a point, with boundary a $C^2$ properly embedded hypersurface $\Sigma$.
\end{theorem}

Given any two distinct points at infinity $x,y \in \partial _\infty \mathbb H ^n$, there exists a $(n-2)-$parameter family of rotations $\{ \mathcal R ^\beta _\theta \}_{\theta \in \mathbb S ^{n-2}}$ that leave invariant $\beta$, where $\beta$ is the complete geodesic in $\mathbb H ^{n}$ joining $x$ and $y$ at infinity.

Moreover, one can check that given any $\theta \in \mathbb S ^{n-2}$ there exist two totally geodesic hyperplanes $P_1$ and $P_2$ such that $\beta  \subset P_1 \cap  P_2$ whose associated hyperbolic reflections $\mathcal R _1 ,\mathcal R _2 \in {\rm Iso}(\mathbb H ^n)$ satisfy
\begin{equation}\label{Eq:Rotations}
\mathcal R^\beta _\theta = \mathcal R _1 \circ \mathcal R_2 .
\end{equation}

As above, one can define

\begin{defn}\label{Def:AxialSymmetric}
Given a $C^2$ function $u : \Omega \to \mathbb R$ is {\bf axially symmetric w.r.t. $\beta$} if there exists a complete geodesic $\beta $ in $\mathbb H ^n$ such that $\mathcal R ^\beta _\theta (\Omega) = \Omega$ for all $\theta \in \mathbb S ^{n-2}$ and
$$ u(p) = u(\mathcal R^\beta _\theta (p)) \text{ for all } \theta \in \mathbb S^{n-2} .$$

When $n=2$, $u$ is axially symmetric if there exists a complete geodesic $\beta $ such that $\mathcal R_\beta (\Omega) = \Omega$ and $u(p)=u(\mathcal R _\beta (p))$ for all $p \in \Omega$, where $\mathcal R_\beta  \in {\rm Iso}(\mathbb H^2)$ is the reflection that leaves invariant $\beta$.
\end{defn}

Also, another consequence of Theorem \ref{ThLR} and Definition \ref{Def:AxialSymmetric} is the following:

\begin{theorem}\label{ThDelaunay}
Assume that $\Omega$ is a domain in $\mathbb{H}^{n}$, with boundary
a $C^2$ properly embedded hypersurface $\Sigma$ and whose asymptotic
boundary consists in two distinct points $x, y \in \mathbb S
^{n-1}$, $x\neq y$, on which the OEP (\ref{1.4}) has a solution
$u\in{C}^{2}(\overline{\Omega})$. Then $\Omega$ is rotationally
symmetric with respect to the axis given by the complete geodesic
$\beta $ whose boundary at infinity is $\{ x,y\}$, i.e., $\beta ^+ =
x$ and $\beta ^- = y$. In other words, $\Omega$ is invariant by the
$(n-2)-$parameter group of rotations in $\mathbb H^n$ fixing
$\beta$. Moreover, $u$ is axially symmetric w.r.t. $\beta$.
\end{theorem}

Note that Theorem \ref{ThLR} and Theorem \ref{OneEnd} prove the BCN-conjecture in $\mathbb H ^n$ under assumptions on its boundary at infinity.

\begin{corollary}\label{CorBCNn}
Assume that $\Omega$ is a domain in $\mathbb{H}^{n}$, with boundary a $C^2$ properly embedded hypersurface $\Sigma$ and whose asymptotic boundary consists at most in one point $x\in \mathbb S ^{n-1}$, on which the OEP (\ref{1.4}) has a solution $u\in{C}^{2}(\overline{\Omega})$. Then,
\begin{itemize}
\item either $\Omega$ is a geodesic ball and $u$ is radially symmetric,
\item or $\Omega$ is a horoball and $u$ is horospherically symmetric.
\end{itemize}
\end{corollary}

\subsection{Graphical properties of the $f$-extremal domain}

We will assume that our $f$-extremal domain is unbounded, since otherwise Theorem \ref{ThLR} implies that $\Omega $ is a geodesic ball.

Assume that $\Omega$ is an unbounded open connected domain in $\mathbb{H}^{n}$ whose boundary is of class $C^{2}$ and on which the OEP (\ref{1.4}) holds. Let $P$ be an oriented totally geodesic hyperplane which interests $\Omega$. So, $P$ divides $\mathbb{H}^{n}\backslash P$ into two connected components, and these two components are classified to be the interior set, denoted by $\mathrm{int}_{\mathbb{H}^{n}}(P)$, and the exterior set, denoted by $\mathrm{ext}_{\mathbb{H}^{n}}(P)$, of $P$, respectively.

Assume that $P $ intersects $\Omega$. Now, we can give a geometric property for bounded connected components of $\Omega\cap\mathrm{ext}_{\mathbb{H}^{n}}(P)$ or $\Omega\cap\mathrm{int}_{\mathbb{H}^{n}}(P)$ as follows.

\begin{theorem} \label{theorem3.5-1}
Assume that $\Omega\cap\mathrm{ext}_{\mathbb{H}^{n}}(P)$ has a bounded connected component $C$. Then the closure of $\partial{C}\cap\mathrm{ext}_{\mathbb{H}^{n}}(P)$ is a graph over $\partial{C}\cap P$.
\end{theorem}

Before to proceed with the proof, we will explain the meaning of graph in the hyperbolic context.

Fix $x,y \in \mathbb S^{1} \equiv \mathbb H ^n _\infty$ two distinct points at the boundary at infinity. Let $\beta : \mathbb R \to \mathbb H^n$ be the unique geodesic joining $x$ and $y$, i.e., $\beta ^+ =x$ and $\beta ^- = y$. Consider the one parameter family of isometries of $\mathbb H ^n$ given by hyperbolic translations at distance $t$ fixing $\beta$, i.e., $\mathcal T ^{\beta}_t : \mathbb H ^n \to \mathbb H ^n$ such that $\mathcal{T}^\beta _t (\beta ) = \beta$ for all $t \in \mathbb R$. Then, since $\{ \mathcal T^\beta_t  \} _{t \in \mathbb R}$ is a one parameter family of isometries, it induces a unit Killing vector field $X^\beta \in \mathcal{X}(\mathbb H^n) $, Moreover, for any totally geodesic hyperplane $P$ such that $x,y \not \in \partial _\infty P$, $\{ \mathcal T ^{\beta}_t  \} _{t \in \mathbb R}$ induces a foliation by totally geodesic hyperplanes in $\mathbb H^n$ given by $P(t) = \mathcal T ^\beta _t (P)$, $t \in \mathbb R$.

Given a totally geodesic hyperplane $P$, there exists a unique complete geodesic $\beta : \mathbb R \to \mathbb H^n$ such that $X^{\beta}(p)$ is orthogonal to $T_{p} P $ for all $p \in P$.

We say that $\Sigma \subset \mathbb H ^n$ is a graph over $P$ if there exists a connected domain, $K \subset P$, and a function $u : K \to \mathbb R$ such that
$$ \Sigma = \{ \mathcal T^{\beta}_{u(p)} (p) \, : \, \, p \in K\} . $$

\begin{proof}[Proof of Theorem  \ref{theorem3.5-1}]

From the explanation above, for a given totally geodesic hyperplane
$P$ and two distinct points $x$, $y$ at the boundary at infinity, if
 $x,y
\not \in \partial _\infty P$, then for the unique geodesic $\beta$
joining these two points with $\beta ^+ =\beta(+\infty)=x$ and
$\beta ^- =\beta(-\infty)= y$, a foliation $P(t) = \mathcal T ^\beta
_t (P)$, $t \in \mathbb R$, which is orthogonal with $\beta$ for any
$t\in(-\infty,+\infty)$, can be built along $\beta$. Moreover,
$P=P(0)=\mathcal T ^\beta _0 (P)$. Since
$C\subseteq\mathrm{ext}_{\mathbb{H}^{n}}(P)$ and it is bounded,
there exists $t_{0}>0$ such that
\begin{eqnarray*}
P(t)\cap C=\emptyset, \qquad \mathrm{for~all}~t\geqslant t_0.
\end{eqnarray*}
Hence, decreasing $t$ we will find some $\bar t$ which is a first
moment such that $P(\bar t)\cap\bar C\neq\emptyset$ and
$P(t)\cap\bar C=\emptyset$ for any $t>\bar t$. Therefore, since
$\bar C$ is compact, there exists $\epsilon>0$ such that $\partial
C_{t}^{+}:=\partial{C}\cap\mathrm{ext}_{\mathbb{H}^{n}}(P)$ is a
graph over $P(t)$, $t\in(\bar t,\bar t+\epsilon)$. This \emph{claim}
follows from the Alexandrov reflection technique introduced in Theorem
\ref{ThLR}. Let us explain this. As we did in Theorem \ref{ThLR},
define
\begin{eqnarray*}
&&C_{t}^{+}:=\partial{C}\cap\mathrm{ext}_{\mathbb{H}^{n}}(P),\\
&&C_{t}^{-}:=\partial{C}\cap\mathrm{int}_{\mathbb{H}^{n}}(P),\\
&&\tilde{C}_{t}^{+}=\mathcal{R}_{t}(C_{t}^{+}),
\end{eqnarray*}
where $\mathcal{R}_{t}$ is the reflection through $P(t)$. Hence,
there exists $\epsilon>0$ small enough such that
\begin{eqnarray*}
\tilde{C}_{t}^{+}\subset C_{t}^{-}, \qquad \forall{t}\in(\bar
t-\epsilon,\bar t),
\end{eqnarray*}
which implies that that $\partial C_{t}^{+}$ is a graph, in the
sense defined above, over $P(t)$. Now, decreasing $t$ up to $0$.
Note that \emph{if $\tilde{C}_{t}^{+}\subset C_{t}^{-}$ for any
$t\in(0,\bar t]$, then
$\partial{C}\cap\mathrm{ext}_{\mathbb{H}^{n}}(P)$ will be a graph
over $P$ and the proof finishes}. Assume this is not true, then
following the ideas in Theorem \ref{ThLR} two situations could
happen:

(1) There exists $t'\in(0,\bar t)$ such that $P( t')$ is orthogonal to $\partial C$ at some point $q\in P(t ')\cap\partial C$;

(2) There exists $t'\in(0,\bar t)$ such that $\partial\tilde{C}_{t'}^{+}$ is internally tangent to $\partial
C_{t'}^{-}$.

In any of the above two cases, applying the maximum principle,
either at the boundary or at the interior, as we did in Theorem
\ref{ThLR}, we will obtain that $C$ is symmetric w.r.t. the totally
geodesic hyperplane $P(t ')$. But this is impossible.

Therefore, $\partial{C}\cap\mathrm{ext}_{\mathbb{H}^{n}}(P)$ is a
graph over $\partial{C}\cap P$.
\end{proof}

Moreover,  Theorem \ref{theorem3.5-1} and its proof let us claim the following four conclusions.

\begin{corollary}
 $C\cap P$ is connected.
\end{corollary}

\begin{corollary}  \label{corollary3.11-1}
The closure of $\partial{C}\cap\mathrm{ext}_{\mathbb{H}^{n}}(P)$ is not orthogonal to $P$ at any point in the sense of the hyperbolic metric $g_{-1}$.
\end{corollary}

\begin{proof}
If the closure $\partial{C}\cap\mathrm{ext}_{\mathbb{H}^{n}}(P)$ were orthogonal to $P$, then $\Omega$ is symmetric w.r.t. $P$, which contradicts the fact that $\Omega$ is unbounded.
\end{proof}

\begin{corollary} \label{corollary3.12-1}
 If $C'$ is the reflection of $C$ through $P$, then the closure of $C\cup{C'}$ stays within $\overline{\Omega}$.
\end{corollary}

\begin{corollary}  \label{corollary3.13-1}
Let $\Omega$ be a $f$-extremal domain of the OEP (\ref{1.4}) satisfying the property
\begin{quote}
{\bf $\mathbb{P}_{2}:$}  There exists a constant $R$ such that $\overline{\Omega}$ does not contain any closed ball of radius $R$.
\end{quote}

Then it is impossible to construct a half-ball of radius $R$ centered at some point in $\partial{C}\cap P$ and staying within $C$.
\end{corollary}

\begin{proof}
Suppose it were possible to construct a half-ball of radius $R$ centered at some point in $\partial{C}\cap P$ and staying within $C$. Then, by Corollary \ref{corollary3.12-1}, the closure of $C\cup{C'}$ would contain a closed ball of radius $R$ centered at some point in $\partial{C}\cap P$, which contradicts property $\mathbb{P}_{2}$. Therefore, our assumption is not true.
\end{proof}


By Lemma \ref{proposition1}, we know that for the OEP (\ref{1.3}), if $f$ satisfies the property $\mathbb{P}_{1}$, then its $f$-extremal domain $\Omega$ has the property $\mathbb{P}_{2}$. Together with Corollary \ref{corollary3.13-1}, we can easily get the following.

\begin{corollary}
Let $\Omega$ be a $f$-extremal domain of the OEP (\ref{1.3}) and assume that $\Omega\cap\mathrm{ext}_{\mathbb{H}^{n}}(P)$ has a bounded connected component $C$.

If the function $f$ in the OEP (\ref{1.3}) satisfies the property $\mathbb{P}_{1}$, then it is impossible to construct a half-ball of radius $R_{\lambda,n}$, which is determined by (\ref{2.2}), centered at some point in $\partial{C}\cap P$ and staying within $C$.
\end{corollary}

\subsection{Concluding remarks}

It is interesting to highlight here the similarities between OEP in $\mathbb H^n$, properly embedded CMC-hypersurfaces and the singular Yamabe Problem.

We already have pointed out that Theorem \ref{ThLR} is the OEP counterpart of the Levitt-Rosenberg's Theorem \cite{LR} for properly embedded CMC-hypersurfaces. From the point of view of the singular Yamabe Problem, V. Bonini, J.M. Espinar and J. Qing extended this for the fully nonlinear elliptic singular Yamabe Problems (cf \cite[Main Theorem B]{BEQ}).

For the sake of clarity, we will explain here what we mean for {\it fully nonlinear elliptic singular Yamabe Problems}. First, we introduce the conformally invariant elliptic PDE in the context of our discussions.  Denote
$$
{\cal C} := \{(x_1, \cdots , x_n) \in \mathbb R ^n : x_i < 1/2, i = 1, \cdots, n\}
$$
and
$$
\Gamma_n := \{(x_1,\cdots, x_n ): x_i >0, i=1, 2, \cdots, n\}.
$$
Consider a symmetric function $f(x_1, \cdots, x_n)$ of $n$-variables with $f(\lambda_0, \lambda_0, \cdots, \lambda_0) = 0$ for some number $\lambda_0 < \frac 12$ and
$$
\Gamma = \ \text{an open connected component of }\{(x_1, \cdots, x_n): f(x_1, \cdots, x_n) > 0\}
$$
satisfying

\begin{itemize}
\item [(1)] $ (\lambda, \lambda, \cdots, \lambda) \in \Gamma \cap {\cal C}, \forall \ \lambda \in (\lambda_0, \frac 12) $,

\item [(2)] $\forall \ (x_1, \cdots, x_n)\in \Gamma\cap{\cal C}$, $\forall \ (y_1, \cdots, y_n)\in\Gamma\cap{\cal C}\cap ((x_1, \cdots, x_n) + \Gamma_n)$, $\exists$ a curve $\gamma$ connecting $(x_1, \cdots, x_n)$ to $(y_1, \cdots, y_n)$ inside $\Gamma\cap{\cal C}$ such that $\gamma' \in \Gamma_n$ along $\gamma$,

\item [(3)] $f \in C^1(\Gamma)$  and  $\frac{\partial f }{ \partial x_i} >0$ in $\Gamma$.
\end{itemize}

Suppose $g = e^{2\rho}g_{0}$ is a complete conformal metric on a domain $\Omega$ of $(\mathbb S^n , g_0)$ satisfying
\begin{equation}\label{EqYamabe}
\begin{matrix}
\hspace{4cm} & f(\lambda _1 , \ldots , \lambda _n) = C  \ \text{and} \ (\lambda _1 , \ldots , \lambda _n) \in \Gamma\cap \mathcal C \ \text{in} \ \Omega, & \hspace{3cm}
\end{matrix}
\end{equation}
for some nonnegative constant $C$, where $(\lambda _1 , \ldots , \lambda _n)$ is the set of eigenvalues of the Schouten curvature tensor of the metric $g$. We refer to equation \eqref{EqYamabe} as the conformally invariant elliptic problem of the conformal metrics on the domain $\Omega$. In particular, taking $f(\lambda _1 , \ldots , \lambda _n) = \lambda _1 + \cdots + \lambda _n$, we obtain the classical singular Yamabe Problem.

Hence, this shows the intimate relationship between Theorem \ref{ThLR}, Levitt-Rosenberg's Theorem \cite{LR} and Bonini-Espinar-Qing's Theorem \cite{BEQ}. Moreover, Levitt-Rosenberg's Theorem has two fundamental consequences in the theory:

\begin{quote}
{\bf Lawson-do Carmo Theorem \cite{dCL}:} {\it The only properly embedded hypersurface $\Sigma \subset \mathbb H^n$ of CMC $H\geq 1$ whose boundary at infinity is a single point is the horosphere. In particular, $H=1$. In other words, there is no properly embedded hypersurface with CMC $H>1$ in $\mathbb H^n$ whose boundary at infinity is a single point.}
\end{quote}

And

\begin{quote}
{\bf Hsiang Theorem \cite{Hs}: }{\it Let $\Sigma \subset \mathbb H ^n$ be a properly embedded CMC hypersurface whose boundary at infinity consist in two distinct points. Then, $\Sigma$ is invariant by the one parameter group of rotations in $\mathbb H^n$ fixing $x$ and $y$.}
\end{quote}

In \cite{BEQ}, the authors obtained the analogous results to the do Carmo-Lawson's Theorem and Hsiang's Theorem in the context of fully nonlinear singular Yamabe Problems. For OEP, Theorem \ref{theorem3.6-1}, Theorem \ref{OneEnd} and Theorem \ref{ThDelaunay} give us the analogous results.

Hence, this suggests that:

\begin{quote}
{\it Any problem for either OEP in $\mathbb H^n$, CMC-hypersurfaces in $\mathbb H^n$ or Singular Yamabe Problems in $\mathbb S ^n$ has a counterpart in each category.}
\end{quote}

\section{Berestycki-Caffarelli-Nirenberg Conjecture in $\mathbb H^{2}$}

From now on in this section, we will focus on the two dimensional case, $\Omega \subset \mathbb H^2$. Here, we will work on the Poincar\'{e} ball model of $\mathbb H ^2$, that is, $(\mathbb{D},g_{-1})=\mathbb{H}^{2}$.

Let $\Omega$ be an unbounded open connected domain in $\mathbb{H}^{2}$, with a boundary of class $C^{2}$, on which the OEP (\ref{1.4}) has a solution $u\in{C}^{2}(\overline{\Omega})$. Moreover, we assume that $\Omega$ has the property $\mathbb{P}_{2}$ mentioned in Corollary \ref{corollary3.13-1}. Under this conditions, we can prove the Berestycki-Caffarelli-Nirenberg Conjecture in $\mathbb H ^2$.

\begin{theorem}\label{ThBCNP2}
Let $\Omega \subset \mathbb H^2$ a domain with properly embbeded $C^2$ connected boundary such that $\mathbb H^2 \setminus \overline{\Omega}$ is connected. If there exists a (strictly) positive function $u\in{C}^{2}(\Omega)$ that solves (\ref{1.4}). If $\Omega$ has the property $\mathbb{P}_{2}$ mentioned in Corollary \ref{corollary3.13-1}, then $\Omega$ must be a geodesic ball and $u$ is radially symmetric.
\end{theorem}
\begin{proof}
In $\mathbb H^2$, if $\partial _\infty \Omega$ has more than two components, then $\mathbb H ^2 \setminus \overline{\Omega}$ disconnects $\mathbb H^2$. Hence, $\partial _\infty \Omega$ has either only one component or none.

If $\partial _\infty \Omega = \O$, then Theorem \ref{ThLR} implies that $\Omega$ is a geodesic ball and $u$ is radially symmetric.

If $\partial_\infty \Omega$ has one component, such a component must be a single point by Lemma \ref{horosphere}. Thus, Theorem \ref{ThLR} would imply that $\Omega$ is a horoball. Then $\Omega$ being a horoball will contradict Theorem \ref{ThBoundary}. This finishes the proof.
\end{proof}

In particular, if the OEP (\ref{1.4}) is replaced by the OEP (\ref{1.3}), and the function $f$ in (\ref{1.3}) has the property $\mathbb{P}_{1}$, then by Proposition \ref{proposition1} we know that the $f$-extremal domain $\Omega$ of the OEP (\ref{1.3}) has the
property $\mathbb{P}_{2}$, which implies that Theorem \ref{ThBCNP2}
still holds in this replacement.

\begin{theorem}\label{BCNConjecture}
Let $\Omega \subset \mathbb H^2$ a domain with properly embbeded connected $C^2$ connected boundary such that $\mathbb H^2 \setminus \overline{\Omega}$ is connected. If there exists a (strictly) positive function $u\in{C}^{2}(\Omega)$ that solves the equation
\begin{eqnarray*}
\left\{
\begin{array}{llll}
\Delta{u}+f(u)=0  \quad&\mathrm{in}\quad ~~\Omega,\\
u>0  \quad&\mathrm{in}\quad ~~\Omega,\\
u=0  \quad&\mathrm{on}\quad \partial\Omega,\\
\langle\nabla{u},\vec{v}\rangle_{\mathbb H ^2}=\alpha \quad&\mathrm{on}\quad
\partial\Omega,
\end{array}
\right.
\end{eqnarray*}where $f:(0,+\infty)\rightarrow\mathbb{R}$ is a Lipschitz function, then $\Omega$ is either a geodesic ball or a horoball.

Furthermore, if $f:(0,+\infty)\rightarrow\mathbb{R}$ satisfies the property $\mathbb{P}_{1}$ mentioned in Lemma \ref{proposition1} for some constant $\lambda$ satisfying $\lambda>\frac{1}{4}$, then $\Omega$ must be a geodesic ball and $u$ is radially symmetric.
\end{theorem}

\subsection{Cylindrically boundedness}

When we are dealing with $f$-extremal domain in $\mathbb H ^2$, the graphical properties developed in Subsection 3.3 will imply the cylindrically boundedness of ends of $\Omega$ that are topologically a half strip $[0,1]\times (0,+\infty)$. An end $E\subset \Omega $ is topologically a half-strip if there exists a compact set $K \subset \mathbb H^2$ such that $E$ is a connected component of $\Omega \setminus K $ and there exists a homeomorphism $h : [0,1]\times (0,+\infty) \to E$.

\begin{remark}
This is the counterpart in OEP of being a properly embbeded annulus for CMC hypersurfaces.
\end{remark}

By using a similar method to that in the proof of \cite[Lemma 6.1]{rs}, which follows geometric ideas in \cite{egr}, we can bound the maximum distance that a bounded connected component $C \subset \Omega \cap \mathrm{ext}_{\mathbb{H}^{2}}(\delta)$ can attain to $\delta $. Specifically,

\begin{lemma} \label{lemma4.1}
Let $\Omega$ be a $f$-extremal domain of the OEP (\ref{1.4}) satisfying the property $\mathbb{P}_{2}$.

Let $C$ be a bounded connected component of $\Omega\cap\mathrm{ext}_{\mathbb{H}^{2}}(\delta)$, and let $h(C)$ be the maximal distance of $\partial{C}$ to $\delta$ in the sense of the hyperbolic metric $g_{-1}$. Then we have $h(C)\leqslant3R$.
\end{lemma}

The proof is a clever use of the reflection technique and using the condition that there is no ball of a certain radius inside. The proof in the hyperbolic case mimic that of the Euclidean case, with the obvious use of the reflection technique developed in Theorem \ref{ThLR}.

Moreover, this Lemma \ref{lemma4.1} is not that fundamental in the hyperbolic setting. It will implies that

\begin{lemma}\label{LemCylin}
Let $\Omega$ be a $f$-extremal domain of the OEP (\ref{1.4}) satisfying the property $\mathbb{P}_{2}$.

Any end $E$ of $\Omega$ which is topologically a half-strip must be cylindrical bounded. That is, there exists a geodesic $\gamma $ and a uniform constant $C$ (depending only on $R$ and $\partial E \cap K$), such that
$$ d(p ,\gamma ) \leq C \text{ for all } p \in E ,$$here $d$ is the distance w.r.t. the hyperbolic metric $g_{-1}$.
\end{lemma}

Another way to see Lemma \ref{LemCylin} is saying that,

\begin{lemma}\label{LemE}
Let $\Omega$ be a $f$-extremal domain of the OEP (\ref{1.4}) satisfying the property $\mathbb{P}_{2}$. The boundary at infinity of an end $E$ which is topologically a half-strip must be a single point and, moreover, such a point at infinity must be a conical point of radius $r$ uniformly bounded by $R$ and $\partial E \cap K$.
\end{lemma}

This Lemma \ref{LemE} is fundamental in the Euclidean case (cf. \cite{rs}). However, in the Hyperbolic setting we already have Theorem \ref{ThBoundary} which implies Lemma \ref{LemE}. As we pointed out, the Hyperbolic geometry imposses more restrictions than the Euclidean geometry. Nevertheless, we think it is important to address these properties for future applications.

\subsection{Concluding remarks}

In dimension $2$ we think it must hold:

\begin{quote}
{\bf BCN-Conjecture in $\mathbb H ^2$: } {\it Let $\Omega \subset \mathbb H ^2$ be a domain with properly embedded $C^2$ boundary and such that $\mathbb H ^2 \setminus \overline \Omega $ is connected. If there exists a (strictly) positive bounded function $u\in{C}^{2}(\Omega)$ that solves the equation
\begin{eqnarray*}
\left\{
\begin{array}{llll}
\Delta{u}+f(u)=0  \quad&\mathrm{in}\quad ~~\Omega,\\
u>0  \quad&\mathrm{in}\quad ~~\Omega,\\
u=0  \quad&\mathrm{on}\quad \partial\Omega,\\
\langle\nabla{u},\vec{v}\rangle_{\mathbb H ^2}=\alpha \quad&\mathrm{on}\quad
\partial\Omega,
\end{array}
\right.
\end{eqnarray*}where $f:(0,+\infty)\rightarrow\mathbb{R}$ is a Lipschitz function, then $\Omega$ must be either
\begin{itemize}
\item a geodesic ball or,
\item a horoball or,
\item a half-space determined by a complete geodesic or,
\item a half-space determined by a complete equidistant curve, i.e., a complete curve of constant geodesic curvature $k_g \in (0,1)$, or,
\item the complement of one of the above examples.
\end{itemize}}
\end{quote}

\section{A height estimate}   \label{section4}
\renewcommand{\thesection}{\arabic{section}}
\renewcommand{\theequation}{\thesection.\arabic{equation}}
\setcounter{equation}{0} \setcounter{maintheorem}{0}

From now on in this section, we will focus on the two dimensional case, $\Omega \subset \mathbb H^2$. Let $\Omega$ be an unbounded open connected domain in
$\mathbb{H}^{2}$, with a boundary of class $C^{2}$, on which the OEP
(\ref{1.4}) has a solution $u\in{C}^{2}(\overline{\Omega})$. Let
$R_{\lambda,n}$, determined by (\ref{2.2}) with $n=2$, be the radius
of the geodesic ball $B_{\mathbb{H}^{2}}(p,R_{\lambda,n})$ on which
the first Dirichlet eigenvalue of the Laplacian is $\lambda$ (i.e.,
$\lambda_{1}\left(R_{\lambda,n}\right)=\lambda$), and let $v$ be a
chosen eigenfunction of $\lambda_{1}\left(R_{\lambda,n}\right)$ such
that
\begin{eqnarray*}
\langle\nabla{v},\vec{v}\rangle_{\mathbb{H}^{2}}=\alpha.
\end{eqnarray*}
For this moment, we assume that $\alpha\neq0$. Set
\begin{eqnarray*}
h_{0}:=\max\limits_{B_{\mathbb{H}^{2}}(p,R_{\lambda,n})}v=v(p).
\end{eqnarray*}
Clearly, $h_{0}$ depends on $\alpha$ and $\lambda$. The last
equality in the above expression holds since $v$ is a radial
function and is decreasing along the radial direction.

By applying a similar method to that in the proof of \cite[Propostion 5.1]{rs} that follows geometric ideas developed in \cite{egr}. We can prove the following.

\begin{proposition}  \label{proposition3.15-1}
Assume that an unbounded open connected $f$-extremal domain $\Omega$
of the OEP (\ref{1.4}) satisfies the property $\mathbb{P}_{2}$
mentioned in Corollary \ref{corollary3.13-1} (here for coherence of
the usage of notations in the argument of this subsection, we use
$R_{\lambda,n}$ to replace $R$ in the property $\mathbb{P}_{2}$).
Let $\Omega'$ be a connected component of
\begin{eqnarray*}
\left\{x\in\Omega|u(x)>h_{0}\right\}.
\end{eqnarray*}
Then the diameter of $\Omega'$ is smaller than $2R_{\lambda,n}$.
\end{proposition}

\begin{proof}
Suppose first that $\Omega'$ is bounded. Let $d$ be the diameter of
$\Omega'$, and suppose that $d\geqslant2R_{\lambda,n}$. As we have
pointed out in Subsection 3.1, $\mathbb{H}^{2}\times\mathbb{R}$ can
be represented by
$\{(\xi_{1},\xi_{2},t)\in\mathbb{R}^{3}|\xi_{1}^{2}+\xi_{2}^{2}<1\}$
with the metric $\langle\cdot,\cdot\rangle:=g_{-1}+dt^{2}$, and the
one-to-one correspondence between $\mathbb{H}^{2}$ and $\mathbb{D}$
is given by a stereographic projection $\mathcal{S}$. Clearly,
$\mathcal{S}$ maps a bounded domain on $\mathbb{H}^{2}$ into a
bounded domain contained in $\mathbb{D}$ without intersecting
$\mathbb{S}_{\infty}^{1}$. Since $\Omega$ is unbounded, the image of
$\Omega$ under the mapping $\mathcal{S}$, which by the abuse of
notations we also denote by $\Omega$, must have at least one
boundary point $q^{\ast}$ at infinity, that is,
$q^{\ast}\in\mathbb{S}_{\infty}^{1}\cap\Omega$. Let $q_{1}$, $q_{2}$
be two points in $\overline{\Omega'}$ such that the hyperbolic
distance between $q_{1}$ and $q_{2}$ is $d$, and $\ell$ a curve in
$\overline{\Omega'}$ joining $q_{1}$ and $q_{2}$ (of course, if
$\Omega$ is regular, $\ell$ can be taken in its boundary). Since the
hyperbolic distance between $q_{1}$ and $q_{2}$ is $d$, there exists
a complete geodesic $L_{1}$ passing through $q_{1}$ and $q_{2}$, and
the part of $L_{1}$ connecting $q_{1}$ and $q_{2}$ is contained in
$\overline{\Omega'}$ completely, and we denote this part by
$\widehat{q_{1}q_{2}}$. Clearly, the length of
$\widehat{q_{1}q_{2}}$ is $d$. Let $m$ be the midpoint of the curve
$\widehat{q_{1}q_{2}}$, and let $L_{2}$ be the complete geodesic
passing through $m$ and orthogonal to $L_{1}$.  Set
$\Gamma=(L_{1}\backslash\widehat{q_{1}q_{2}})\cup\ell$. Clearly,
$\Gamma$ divides $\mathbb{D}\backslash\Gamma$ into two connected
components, and we denote them by $H_{1}$ and $H_{2}$ respectively.
Let $p\in{L}_{2}\cap{H_{2}}$ be a point very far way from $\Omega'$
in the sense of the hyperbolic metric $g_{-1}$. Now, consider the
graph $G$ of the eigenfunction $v$ defined on
$B_{\mathbb{H}^{2}}(p,R_{\lambda,n})$ by (\ref{2.2}) with the
Neumann boundary value
$\langle\nabla{v},\vec{v}\rangle_{\mathbb{H}^{2}}=\alpha$, and move
the point $p$ along the complete geodesic $L_{2}$ towards $\Omega'$.
Since the length of $\widehat{q_{1}q_{2}}$ is
$d\geqslant2R_{\lambda,n}$, $u(x)>h_{0}$ for $x\in\Omega'$, and
$u=0$ on $\partial\Omega$, there will exist a first contacting point
between the moved graph $G$ and the graph of $u$ over $\Omega$ at
some interior point of $\Omega$ or the boundary of $\Omega$. Both
situations are impossible by applying the Hopf maximum principle
(both the interior and the boundary versions). So, our assumption is
not true, which implies that $d<2R_{\lambda,n}$ for the case that
$\Omega'$ is bounded.

Suppose now that $\Omega'$ is unbounded, there exists a divergent
curve $\gamma(t)$, contained in $\Omega'$ with
$\lim\limits_{t\rightarrow-\infty}\gamma(t)=\lim\limits_{t\rightarrow+\infty}\gamma(t)=q^{\ast}$,
such that an arc $\ell\subset\gamma(t)$ satisfies the property that
the boundary points $q_{1}$, $q_{2}$ of $\ell$ have a hyperbolic
distance greater than and equal to $2R_{\lambda,n}$. Then one can
repeat the above argument to get a contradiction. This completes the
proof.
\end{proof}

\begin{remark}
\rm{By \cite[Remark 5.2]{rs} and the method in the proof of
Proposition \ref{proposition3.15-1}, it is easy to explain that the
Neumann boundary data $\alpha$ cannot vanish. }
\end{remark}

We can obtain the following result by applying Proposition \ref{proposition3.15-1} directly.

\begin{theorem}
Let $\Omega$ be an unbounded open connected $f$-extremal domain $\Omega$ of the OEP (\ref{1.3}), and let $\Omega'$ be a connected
component of
\begin{eqnarray*}
\left\{x\in\Omega|u(x)>h_{0}\right\}.
\end{eqnarray*}
If the function $f$ satisfies the property $\mathbb{P}_{1}$ mentioned in Lemma \ref{proposition1}, then the diameter of $\Omega'$ is smaller than $2R_{\lambda,n}$.
\end{theorem}

\section*{Acknowledgments}
\renewcommand{\thesection}{\arabic{section}}
\renewcommand{\theequation}{\thesection.\arabic{equation}}
\setcounter{equation}{0} \setcounter{maintheorem}{0}

J. Mao would like to thank Dimas Percy Abanto Silva for many useful
discussions.  Also, he would like to thank IMPA for their support
during the preparation of this paper. J. Mao is partially supported
by CNPq-Brazil (Grant No. 150033/2014-1) and the NSF of China (Grant
No. 11401131).

The first author, Jos\'{e} M. Espinar, is partially supported by Spanish MEC-FEDER Grant MTM2013-43970-P; CNPq-Brazil Grants 405732/2013-9 and 14/2012 - Universal, Grant 302669/2011-6 - Produtividade;  FAPERJ Grant 25/2014 - Jovem Cientista de Nosso Estado.

 \end{document}